\documentclass[a4paper,10pt]{article}

\usepackage[utf8]{inputenc}
\usepackage[T1]{fontenc}
\usepackage[english]{babel}

\usepackage[inline]{enumitem} 
\newlist{steps}{enumerate}{1}
\setlist[steps, 1]{wide=0pt, label=Step \arabic*., font=\scshape}

\usepackage{authblk} 
\usepackage{cancel}
\usepackage{caption}
\usepackage{subcaption}
\usepackage{tabularx} 

\usepackage[margin=2.4cm]{geometry}

\usepackage{amsmath,amssymb,amsfonts,amsthm,amsbsy,stmaryrd,mathrsfs}
\usepackage{bbold} 
\usepackage{nccmath} 
\usepackage{units}   

\usepackage{graphicx}
\usepackage{xcolor}

\usepackage{hyperref}

\usepackage[capitalise]{cleveref} 

\newcommand{\eps}{\varepsilon}
\newcommand{\jsize}{\xi} 
\DeclareMathOperator{\sgn}{sgn}
\newcommand{\dd}{\mathop{}\!\mathrm{d}}
\newcommand{\abs}[1]{\left\lvert#1\right\rvert}
\newcommand{\absd}[1]{\left\|#1\right\|}
\renewcommand{\textbf}[1]{\begingroup\bfseries\mathversion{bold}#1\endgroup} 
\newcommand{\citer}[2][]{#1 in \cite{#2}} 
\newcommand{\commenter}[1]{}

\theoremstyle {definition} 
\newtheorem{ex}{Example}[section]
\theoremstyle{plain} \newtheorem {theo} {Theorem}[section]

\newtheorem{prop}[theo]{Proposition}
\newtheorem{lem}[theo]{Lemma}
\newtheorem{remark}[theo]{Remark}

\begin{document}

\title{Kinetic time-inhomogeneous Lévy-driven model}
    \author{Mihai Gradinaru and Emeline Luirard}
    \affil{\small Univ Rennes, CNRS, IRMAR - UMR 6625, F-35000 Rennes, France\\
    \texttt{\small \{Mihai.Gradinaru,Emeline.Luirard\}@univ-rennes1.fr}}
\date{
}
\maketitle

{\small\noindent {\bf Abstract:}~We study a one-dimensional kinetic stochastic model driven by a Lévy process with a non-linear time-inhomogeneous drift. More precisely, the process $(V,X)$ is considered, where $X$ is the position of the particle and its velocity $V$ is the solution of a stochastic differential equation with a drift of the form $t^{-\beta}F(v)$. The driving process can be a stable Lévy process of index $\alpha$ or a general Lévy process under appropriate assumptions. The function $F$ satisfies a homogeneity condition and $\beta$ is non-negative. The behavior in large time of the process $(V,X)$ is proved and the precise rate of convergence is pointed out by using stochastic analysis tools. To this end, we compute the moment estimates of the velocity process.}\\
{\small\noindent{\bf Keywords:}~kinetic stochastic equation; time-inhomogeneous stochastic differential equation; Lévy process; explosion times; scaling transformations; asymptotic distributions; ergodicity; tightness.}\\
{\small\noindent{\bf MSC2020 Subject Classification:}~Primary~60H10; Secondary~60F17; 60G52; 60G18; 60J65.}

\section{Introduction}
In this paper, we consider a one-dimensional stochastic kinetic model driven by a Lévy process. 
\begin{equation}\label{SKE_stable}
    \dd V_t=\dd L_t-F(V_t)t^{-\beta}\dd t\quad
    \mbox{and}\quad
    X_t=X_0+\int_{0}^{t}V_s\dd s.
\end{equation}
The process $(V_t,X_t)_{t\geq0}$ may be thought of as the velocity and position processes of a particle subject to a friction force $F(v)t^{-\beta}$ and interacting with its environment. \\
Our purpose is to study the long-time behavior of solutions of \eqref{SKE_stable} where $L$ is an $\alpha$-stable (non-symmetric) Lévy process. More precisely, we look for the convergence in distribution of the process
$(V_{t/\eps}, X_{t/\eps})_{t>0}$, as $\eps \to 0$, with an appropriate rate. 

\medskip \noindent
It is a simple observation, when $F=0$, to see that the rescaled process
 $(\eps^{\frac{1}{\alpha}}V_{t/\eps}, \eps^{1+\frac{1}{\alpha}}X_{t/\eps})_{t>0}$ converges in distribution towards the Kolmogorov process $(\mathcal{S}_t, \int_0^t \mathcal{S}_s \dd s)_{t>0}$, where $\mathcal{S}$ has the same distribution as the driving process.\\
The goal of the present paper is to extend the results obtained in \cite{GradinaruAsymptoticbehaviortimeinhomogeneous2021a}, where the driving process is a Brownian motion ($\alpha=2$).

\bigskip
The study of stochastic differential equations (SDEs) driven by a Lévy process is a topic of great interest (see \cite{BassStochasticdifferentialequations2003} for a survey). The $\alpha$-stable perturbation is a generalization of the Gaussian case, and it is also motivated by some Langevin-type models in stochastic climate dynamics (see \cite{DitlevsenObservationastablenoise1999}). So far, most of the papers present results about existence and uniqueness of solution, see for instance \cite{ApplebaumAsymptoticStabilityStochastic2009}, \cite{DongJumpstochasticdifferential2018}, \cite{KurenokStochasticEquationsTimeDependent2007}, \cite{Pilipenkoexistencepropertiesstrong2013}, \cite{ChenStochasticflowsLevy2018} and \cite{ChenWellposednesssupercriticalSDE2017}. The coefficients of the studied SDE are often supposed to be time-homogeneous (see for instance \cite{ApplebaumAsymptoticStabilityStochastic2009} and \cite{DongJumpstochasticdifferential2018}). Accordingly, the case of time-dependent coefficients is scarcely studied (see \cite{ChenWellposednesssupercriticalSDE2017}, \cite{KurenokStochasticEquationsTimeDependent2007} and \cite{ZhangStochasticdifferentialequations2013}). In this situation, the usual tools associated with time-homogeneous equation may no longer be invoked. \\ Furthermore, few papers (see
\cite{ApplebaumAsymptoticStabilityStochastic2009}, \cite{PriolaExponentialergodicityregularity2012}, \cite{RekerShorttimebehaviorsolutions2020}) present results about the asymptotic behavior of the solution of such SDEs. For instance, in \cite{ApplebaumAsymptoticStabilityStochastic2009} the authors give conditions for asymptotic stability of the solutions for a SDE driven by a Brownian motion and a compensated Poisson process, with coefficients that are supposed to satisfy usual global Lipschitz and growth assumptions.  In \cite{PriolaExponentialergodicityregularity2012}, the authors establish the exponential ergodicity of the solutions of a SDE driven by an $\alpha$-stable process, where the drift coefficient is supposed to be the sum of two components, one linear and the other bounded. In these papers, coefficients are time-homogeneous.
In a number of articles, the small noise influence of the solutions is analyzed. To our knowledge, the only works considering the long-time behavior are \cite{FournierOnedimensionalcritical2021}, in a time-homogeneous setting, and the present one.
\bigskip

Let us explain heuristically what the intuition of our analysis is. In long-time regime, we observe three schemes, depending on the balance between the space and time coefficients of the drift function with respect to $\alpha$, the parameter of stability of the driving process. 
When the drag force is sufficiently ``small at infinity'', the convergence towards the Kolmogorov process $(\mathcal{S},\int \mathcal{S})$ still holds.
When the two terms in the stochastic equation of the velocity process offset, we still get a kinetic process of the form $(\mathcal{V}, \int \mathcal{V})$, as limiting process. Though the process $\mathcal{V}$ no longer has the same distribution as the driving Lévy process. Alternatively, when the drift swings with the random noise, the limiting process is no longer kinetic.
\medskip

Proofs are mainly based on moment estimates and the self-similarity of the driving process.
By their scaling property, Lévy stable processes are natural extensions of the Brownian motion. However, the jump component of the Lévy noise brings difficulties. Indeed, by contrast with a Brownian motion, an $\alpha$-stable Lévy process can only have moments of order $\kappa\in [0,\alpha)$. Thus, moment estimation of the velocity process stands as a significant part of our study (see \cref{s:moment_levy}). Moment estimates of Lévy and Lévy-type processes were studied in \cite{LuschgyMomentestimatesLevy2006}, \cite{KuhnExistenceestimatesmoments2017} and \cite{DengshiftHarnackinequalities2015}. Nevertheless, the methods used can not be easily adapted to the solutions of a SDE. In fact, the key idea will be to make a non-homogeneous cutting of the jumps size of the driving process. As explained in \cite{ChaudruDeRaynalMultidimensionalstabledrivenStochastic2020} and references therein, the cutting threshold $\jsize \mapsto \jsize^{\frac{1}{\alpha}}$ makes appear integral terms satisfying the scaling property (see the proof of \cref{esperance3_levy}). \\
The proof of the critical and sub-critical cases (see \cref{main_thm_inh_levy_critique}) significantly relies on a change in both space and time, taking advantage of the scaling property of the driving process, to be close to a stationary time-homogeneous SDE, as performed in \cite{ApplebySolutionsStochasticDifferential2009} and \cite{GradinaruExistenceasymptoticbehaviour2013}. \\
In addition, extensions for the solution of the SDE driven by a general Lévy process are stated in Theorems \ref{gene_thm_levy} and \ref{gene_thm_levy2}. Let us point out that the case analyzed in \cref{ss: withBM} agrees with \citer[the equation (1) p. 1442]{DitlevsenObservationastablenoise1999}, where the described Langevin equation has two noise terms, a white noise and a pure-jump noise. To this end, we first study the asymptotic behavior of the rescaled Lévy driving process $(v_{\eps}L_{t/\eps})_{t\geq 0}$, where $v_{\eps}$ is some rate of convergence. Under appropriate assumptions on its Lévy measure, it converges in distribution, as $\eps$ goes to zero (see \cref{limit_levy} below).

\bigskip

Here is the structure of the paper. In \cref{s:main_results_stable}, we introduce some notations and state our main results. Theorems \ref{main_thm_inh_levy} and \ref{main_thm_inh_levy_critique} are extension for the $\alpha$-stable setting of the main results stated in \cite{GradinaruAsymptoticbehaviortimeinhomogeneous2021a} in case of a Brownian driving process.
We study the existence of the solution of \eqref{SKE_stable} in \cref{s:existence_levy}. In \cref{s:moment_levy}, we give estimates of the moment, which also ensure the non-explosion of the velocity process. The proofs of our main results are presented in \cref{s:proofs_levy}. Finally, \cref{s: extended} is devoted to the study of the convergence of a rescaled Lévy process and to the extension of \cref{main_thm_inh_levy}. 

\section{Notations and main results}\label{s:main_results_stable}
Let $(L_t)_{t\geq0}$ be a Lévy process. Throughout the paper, we deal with $L$ as an $\alpha$-stable Lévy process with $\alpha\in (0,2)$. We call $\nu$ its Lévy measure, given by $\nu(\dd z)= \abs{z}^{-1-\alpha}[a_+\mathbb{1}_{\{z>0\}}+a_-\mathbb{1}_{\{z<0\}}]\dd z$ with $a_+,a_-\geq 0$. As a Lévy measure, it satisfies $\int_{\mathbb{R}^*}(1\wedge z^2)\nu(\dd z)<+\infty$. By Lévy-Itô's decomposition, if $\alpha\in (0,2)$, then $L$ is a pure-jump Lévy process and there exists a Poisson point measure $N$ and its compensated Poisson measure $\widetilde{N}$ such that, for all $t\geq0$, 
\begin{equation}
   L_t= \begin{cases}
        \medint\int_0^t\medint\int_{\mathbb{R}^*}zN(\dd s, \dd z) & \mbox{ if }\alpha\in (0,1),\\
        \medint\int_0^t\medint\int_{\{0<\abs{z}<1\}}z\widetilde{N}(\dd s, \dd z)+\medint\int_0^t\medint\int_{\{\abs{z}\geq 1\}}zN(\dd s, \dd z) & \mbox{ if }\alpha=1,\\
        \medint\int_0^t\medint\int_{\mathbb{R}^*}z\widetilde{N}(\dd s, \dd z) & \mbox{ if }\alpha\in(1,2).
    \end{cases}
\end{equation}
In \cref{s: extended}, the case of a generalized Lévy driving process will be discussed.\\ 
The space of continuous functions ${\mathcal{C}}((0,+\infty), \mathbb{R})$ is endowed with the uniform topology
\[\displaystyle \dd_u:f,g\in \mathcal{C}((0,+\infty))\mapsto \sum_{n=1}^{+\infty}\dfrac{1}{2^n}\min\Big(1,\sup_{[\frac{1}{n},n]}\abs{f-g}\Big).\]
Set $\Lambda:=\{\lambda:\mathbb{R}^+ \to \mathbb{R}^+, \text{continuous and increasing function s.t. } \lambda(0)=0, \ \lim\limits_{t\to +\infty}\lambda(t)=~+\infty  \}$ and  \[k_n(t)=\begin{cases}
        1     & \mbox{if } \frac{1}{n}\leq  t\leq n, \\
        n+1-t & \mbox{if } n<t<n+1,                  \\
        0     & \mbox{if } n+1\leq t.
    \end{cases}  \]
The space of right-continuous with left limits (càdlàg) functions ${\mathcal{D}}((0,+\infty),\mathbb{R})$ is endowed with the Skorokhod topology $\dd_s$ defined for $(f, g)\in \mathcal{D}((0,+\infty),\mathbb{R})^2$ by
\[ 
    \sum_{n= 1}^{+\infty}\dfrac{1}{2^n}\min\left(1, \inf\left\{ a \mbox{ s.t. } \exists \lambda\in \Lambda, \ \sup_{s\neq t}\abs{\log \dfrac{\lambda(t)-\lambda(s)}{t-s}}\leq a, \ \sup_{t \geq\frac1n}\abs{k_n(t)\left(f\circ \lambda(t)-g(t) \right)  } \leq a \right\}\right).\]
For simplicity, we shall write $\mathcal{C}$ and $\mathcal{D}$ for $\mathcal{C}((0,+\infty), \mathbb{R})$ and $\mathcal{D}((0,+\infty),\mathbb{R})$, respectively.
\\
For a family $((Z_t^{(\eps)})_{t> 0})_{\eps>0}$ of càdlàg processes, we write
\[(Z_t^{(\eps)})_{t> 0} \quad \underset{\eps \to 0}{\Longrightarrow} \quad (Z_t)_{t>0}, \]
if $(Z_t^{(\eps)})_{t>0}$ converges in distribution to $(Z_t)_{t>0}$ in ${\mathcal{D}}$, as $\eps \to 0$.\\
We write \[(Z_t^{(\eps)})_{t> 0} \quad  \overset{\text{f.d.d.}}{\underset{\eps \to 0}{\Longrightarrow}}\quad  (Z_t)_{t>0}, \]
if for all finite subsets $S \subset (0,+\infty)$, the vector $(Z_t^{(\eps)})_{t\in S}$ converges in distribution to $(Z_t)_{t\in S}$ in $\mathbb{R}^S$, as $\eps \to 0$.\\
Let $\beta$ a non-negative real number and $F$ a continuous function satisfying
\begin{equation}\tag{$H_{\gamma}$}\label{hyp1_stable}
    \mbox{for some }\gamma \in \mathbb{R},\ \forall v \in \mathbb{R},\ \lambda>0, \ F(\lambda v)=\lambda^{\gamma}F(v).
\end{equation}
We introduce another assumption on $F$, which will sometimes be imposed in the sequel.
\begin{equation}\tag{$H_{\sgn}$}\label{hyp_sign_stable}
    \begin{gathered}
        \mbox{ When \begin{enumerate*}[label=(\roman*)]
    \item $\alpha \in (0,1]$ or
    \item $\alpha \in (1,2)$ and $\gamma \geq 1$,
\end{enumerate*}} \\ \mbox{we suppose furthermore that for all $v\in \mathbb{R}$, $vF(v)\geq 0$.}
    \end{gathered}
\end{equation}
We take an interest into the following one-dimensional stochastic kinetic model defined, for $t\geq t_0>0$, by
\begin{equation}\label{equation time_levy}\tag{\text{SKE}}
    \dd V_t= \dd L_t-t^{-\beta}F(V_{t})\dd t, \mbox{ with } V_{t_0}=v_0>0,\quad
    \mbox{and}\quad
    \dd X_t=V_t \dd t, \mbox{ with } X_{t_0}=x_0\in \mathbb{R}.
\end{equation}
In the following, $\sgn$ is the sign function with the convention that $\sgn(0)=0$. The abbreviation a.s.\ stands for almost surely. We denote by $C$ some positive constants, which may change from line to line. We use the subscripts to indicate the parameters on which it depends. For instance, $C_{t_0,\alpha}$ denotes a constant depending on the parameters $t_0$ and $\alpha$. \\
\begin{remark}\label{pi_stable}
    If a function $\pi$ satisfies \eqref{hyp1_stable}, then for all $x\in \mathbb{R}$,  $\pi(x)=\pi(\sgn(x))\abs{x}^{\gamma}$.
\end{remark}
As an example of a function satisfying \eqref{hyp1_stable}  one can keep in mind $F:v\mapsto \sgn(v)\abs{v}^{\gamma}$ (see also \cite{GradinaruExistenceasymptoticbehaviour2013}).
\\
Let us state our main results.
\begin{theo}\label{main_thm_inh_levy}
    Consider $\gamma \in (1-\frac{\alpha}{2},\alpha)$. Assume that \eqref{hyp1_stable} and \eqref{hyp_sign_stable} are satisfied, and $\beta >1+\frac{\gamma-1}{\alpha}$.
    Let $(V_t,X_t)_{t\geq t_0}$ be the solution to \eqref{equation time_levy} and $(\mathcal{S}_t)_{t\geq0}$ be an $\alpha$-stable process, having same distribution as $(L_t)_{t\geq0}$. \\Then, in the space $\mathcal{D}$,
    \begin{equation*}
        (\eps^{\frac{1}{\alpha} }V_{t/\eps}, \eps^{1+\frac{1}{\alpha}}X_{t/\eps})_{t\geq \eps t_0} \quad \underset{\eps \to 0}{\Longrightarrow}\quad \left(\mathcal{S}_{t}, \int_{0}^t\mathcal{S}_s \dd s\right)_{t>0}.
    \end{equation*} 
\end{theo}
\begin{remark}
    \cref{main_thm_inh_levy} is also true when the following hypothesis holds instead of \eqref{hyp1_stable}.
    \begin{equation}\tag{$H^{\gamma}_{\text{bis}}$}\label{hyp2_levy}
    \begin{gathered}
        F \mbox{ is such that \eqref{equation time_levy} has a unique solution up to explosion and } \\ \abs{F}\leq G\mbox{ where }G \mbox{ is a positive function satisfying \eqref{hyp1_stable}.}
    \end{gathered}
\end{equation}
For instance, the function $F:v\mapsto \frac{v}{(1+v^2)}$ (see also \cite{FournierOnedimensionalcritical2021}) satisfies \eqref{hyp2_levy} (with $\gamma=0$).
\end{remark}
\begin{theo}\label{main_thm_inh_levy_critique}
    Consider $\gamma \in (1-\frac{\alpha}{2},\alpha)$ and $\beta=1+\frac{\gamma-1}{\alpha}$. Assume that \eqref{hyp1_stable} and \eqref{hyp_sign_stable} are satisfied. Let $(V_t,X_t)_{t\geq t_0}$ be the solution to \eqref{equation time_levy}.
    \\
    Call $\widetilde{H}$ the ergodic process solving the following SDE driven by an $\alpha$-stable process $(R_t)_{t\geq0}$ with same distribution as $L$ and starting at its invariant measure,
    \begin{equation*}
        \dd H_s=\dd R_s-\dfrac{H_s}{\alpha}\dd s-F\big(H_s\big)\dd s.
    \end{equation*}
    We denote by $\Lambda_{F, t_1, \cdots, t_d}$ the finite-dimensional distributions of $\widetilde{H}$. We call $(\mathcal{V}_t)_{t\geq0}$ the process having as finite-dimensional distributions the pushforward measure of $\Lambda_{F,\log(t_1), \cdots, \log(t_d)}$ by the linear map $T(u_1, \cdots, u_d):= (t_1^{\nicefrac{1}{\alpha}}u_1, \cdots, t_d^{\nicefrac{1}{\alpha}}u_d)$. 
    \\Then, under \eqref{hyp1_stable}, the following convergence, in the space $\mathcal{D}$, holds
    \begin{equation*}
        (\eps^{\frac{1}{\alpha}}V_{t/\eps}, \eps^{1+\frac{1}{\alpha}}X_{t/\eps})_{t\geq \eps t_0}\quad \underset{\eps \to 0}{\Longrightarrow} \quad\left(\mathcal{V}_{t}, \int_{0}^t\mathcal{V}_s \dd s\right)_{t>0}.
    \end{equation*} 
\end{theo}
\begin{theo}\label{sub-critical_levy}
	Consider $\alpha>1$, $\gamma \in [1,\alpha)$ and $\beta<1+\frac{\gamma-1}{\alpha}$. Assume that \eqref{hyp1_stable} and \eqref{hyp_sign_stable} are satisfied. Let $(V_t,X_t)_{t\geq t_0}$ be the solution to \eqref{equation time_levy}. 
    Define $q:= \frac{\beta}{\alpha+\gamma-1}<\frac{1}{\alpha}$.
	Call $\widehat{H}$ the ergodic process solving the following SDE driven by an $\alpha$-stable process $(R_t)_{t\geq0}$ with same distribution as $L$ and starting at its invariant measure, 
	\begin{equation*}
		\dd H_s=\dd R_s-F\left(H_s\right)\dd s.
	\end{equation*} Call $\Pi_F$ its invariant measure. We call $(\mathscr{V}_t)_{t\geq0}$ the process whose finite dimensional distribution are $T*\left( \Pi_{F}^{\otimes d}\right)$: the pushforward measure of $ \Pi_{F}^{\otimes d}$ by the linear map $T(u_1, \cdots, u_d):=({t_1}^qu_1, \cdots, {t_d}^qu_d)$. \\
	Then,
	\begin{equation*}
		\left(\eps^{q}V_{t/\eps}\right)_{t\geq \eps t_0}\quad \overset{\text{f.d.d.}}{\underset{\eps \to 0}{\Longrightarrow}}\quad \left(\mathscr{V}_{t}\right)_{t\geq0}.
		\end{equation*}
\end{theo}
\begin{remark}
    As we will see in \cref{s:existence_levy}, the assumption $\gamma >1-\frac{\alpha}{2} $ is needed in order to obtain the existence up to explosion of the solution under hypothesis \eqref{hyp1_stable}.
\end{remark}
\begin{remark}\label{explosion_time_levy}
    Let us point out that, during the proof of \cref{main_thm_inh_levy} and \cref{main_thm_inh_levy_critique}, we employ some moment estimates for the solution $V$. We state the estimates below. \\
    Assume that \eqref{hyp_sign_stable} is satisfied. We suppose also that the hypothesis on the sign of $F$ holds for $(\alpha,\gamma,\kappa)\in(1,2)\times [0,1]\times (1,\alpha)$. Then, for any $\alpha\in (0,2), \ \gamma \in \mathbb{R}, \ \beta\in \mathbb{R}$ and $\kappa\in[0,\alpha)$, there exists a constant $C_{\gamma,\kappa,\beta,t_0}$ such that,
    \begin{equation*}
        \forall t\geq  t_0, \ \mathbb{E}\left[ \abs{V_{t}}^{\kappa} \right]\leq C_{\gamma,\kappa,\beta,t_0}t^{\frac{\kappa}{\alpha}}.
    \end{equation*}
Note that the above bounds are the best possible, taking $F=0$. 
\end{remark}

\section{Existence up to explosion} \label{s:existence_levy}
In this section, we study the existence of the solution to \eqref{equation time_levy} up to explosion time.
\begin{remark}\label{Holder_levy}
    Assume that \eqref{hyp1_stable} holds. Pick $\beta\geq0$. If $0<\gamma<1$, the function $x\mapsto F(x)t^{-\beta}$ is $\gamma$-Hölder and if $\gamma\geq 1$, it
    is locally Lipschitz. 
\end{remark}

\begin{prop}\label{existence hyp1_stable}
    Assume that \eqref{hyp1_stable} is satisfied. There exists a pathwise unique strong solution to \eqref{equation time_levy}, defined up to the explosion time, provided that
    \begin{enumerate}[label=(\roman*)]
        \item $1-\frac{\alpha}{2}<\gamma<1$ and $\beta \geq 0$ when $\alpha\in(0,2)$.
        \item $\gamma\geq 1$ when $\alpha>1$.
    \end{enumerate}
\end{prop}
\begin{proof}
    If $\gamma\in(0,1)$, the drift coefficient is $\gamma$-Hölder (see \cref{Holder_levy}) and locally bounded, thereby the conclusion of the first point follows from \citer[Remark 1.3]{ChenWellposednesssupercriticalSDE2017}.\\
    Assume now that $\alpha>1$ and $\gamma\geq1$. The drift coefficient is locally Lipschitz (see \cref{Holder_levy}) and locally bounded, so we can apply \citer[Lemma 115 p. 78]{SituTheoryStochasticDifferential2005} 
    to get the pathwise uniqueness. Thanks to \citer[Theorem 137 p. 104]{SituTheoryStochasticDifferential2005}, it suffices to prove that there exists a weak solution. 
    \\
    The drift coefficient is continuous with respect to its two variables, so it is a locally bounded and measurable function.
    By a standard localization argument, using \citer[Theorem 9.1 p. 231]{IkedaStochasticDifferentialEquations1981}, since the drift coefficient is locally Lipschitz, there is a unique solution defined up to explosion. 
\end{proof}

\section{Moment estimates and non-explosion of the velocity process} \label{s:moment_levy}

In this section, we present estimates on moments of the velocity process $V$ solution to \eqref{equation time_levy}. This will be useful to conclude of the non-explosion of solution to \eqref{equation time_levy} with \cref{explosion_levy}, and to control some terms appearing along the proofs of Theorem \ref{main_thm_inh_levy} and \ref{main_thm_inh_levy_critique} in \cref{s:proofs_levy}.\\
Let $V$ be the unique solution up to explosion time to \eqref{equation time_levy}.
For all $r\geq0$, define the stopping time 
\begin{equation}\label{stopping_time_levy}
    \tau_r:=\inf \{ t\geq t_0, \ \abs{V_t}\geq r\}.
\end{equation}
Set $\tau_{\infty}:= \lim_{r\to +\infty}\tau_r$ the explosion time of $V$.\\
We give first a sufficient condition for the non-explosion of a general process.
\begin{lem}\label{explosion_levy} Let $(Y_t)_{t\geq t_0}$ be a càdlàg process and $\tau_{\infty}$ its explosion time.
    Assume that there exist two measurable and non-negative functions $\phi$ and $b$ such that
    \begin{enumerate}[label=(\roman*)]
        \item $\phi$ is non-decreasing and $\lim_{r\to \infty}\phi(r)=+\infty$,
        \item $b$ is finite-valued,
        \item and for all $t\geq t_0$,
              \begin{equation}\label{borne_explosion_levy}
                  \sup_{r\geq0}\mathbb{E}\left[\phi(\abs{Y_{t\wedge \tau_{r}}})\right] \leq b(t).
              \end{equation}
    \end{enumerate}
    Then $\tau_{\infty}=+\infty$ a.s.
\end{lem}
\begin{proof}
    Pick $t\geq t_0$. Using the definition of $\tau_r$, the monotony of $\phi$ and \textit{(iii)}, we get, for all $r\geq0$,
    \begin{equation*}
        \phi(r)\mathbb{P}(\tau_r\leq t)\leq \mathbb{E}\left[ \phi\left(Y_{\tau_r}\right)\mathbb{1}_{\{\tau_r\leq t\} }\right]\leq \mathbb{E}\left[ \phi\left(Y_{\tau_r}\right)\right]\leq b(t).
    \end{equation*}
    Thus, by Fatou's lemma, 
 \begin{equation*}
     0\leq \mathbb{P}\left(\tau_{\infty}\leq t \right)\leq \liminf_{r\to \infty}\mathbb{P}(\tau_r\leq t) \leq b(t)\lim_{r\to \infty}\dfrac{1}{\phi(r)}=0.
 \end{equation*}
 As a consequence, 
 \[0\leq \mathbb{P}\left(\tau_{\infty}<+\infty \right) \leq \sum_{t\in \mathbb{Q}}\mathbb{P}\left(\tau_{\infty}\leq t \right) =0. \]
This concludes the proof.
\end{proof}
We will show that there exists a constant $C_{\gamma,\kappa,\beta,t_0}$ such that \begin{equation}\label{eq: moment def p}
    \forall t\geq  t_0, \ \mathbb{E}\left[ \abs{V_{t}}^{\kappa} \right]\leq C_{\gamma,\kappa,\beta,t_0}t^{\frac{\kappa}{\alpha}}.
\end{equation}
 
\begin{prop}\label{esperance_levy} Pick $\alpha\in (0,1)$. Assume that \eqref{hyp_sign_stable} holds. Recall that $(V_t)_{t\geq t_0}$ is the solution of \eqref{equation time_levy}.  For any $\gamma$, $\beta$, the explosion time $\tau_{\infty}$ is a.s.\ infinite and for all $\kappa \in [0,\alpha)$, there exists a constant $C_{\kappa,t_0}$ such that, we have
    \begin{equation}\label{eq:moment_alpha_1 }
        \forall t\geq  t_0, \ \mathbb{E}\left[ \abs{V_{t}}^{\kappa} \right]\leq C_{\kappa,t_0}t^{\frac{\kappa}{\alpha}}.
    \end{equation}
\end{prop}
\begin{proof}
    Fix $t\geq t_0$. Since $\alpha<1$, the stable process can be written as 
    \[L_t = \int_0^t\int_{\mathbb{R}^*}zN(\dd s, \dd z) =  \sum_{s\leq t}\Delta L_s.\]
    Fix $\kappa \in [0,\alpha)$.
    Pick the sequence of $\mathcal{C}^2$-functions $f_n:x\mapsto \sqrt{x^2+\frac{1}{n}}$, which converges uniformly to $x\mapsto \abs{x}$ on $\mathbb{R}$.
    Then, for all $n\geq1$, we apply Itô's formula (see \citer[Theorem 32 p. 78]{ProtterStochasticIntegrationDifferential2005}) to get
    \[
        \begin{aligned}
            f_n(V_{t\wedge \tau_{r}}) & =f_n(v_0)-\int_{t_0}^{t\wedge \tau_{r}} f_n'(V_{s})F(V_s)s^{-\beta}\dd s + \int_{t_0}^{t\wedge \tau_{r}}\int_{\mathbb{R}^*}\left(f_n(V_{s-}+z)-f_n(V_{s-})\right)N(\dd s, \dd z) \\ &\leq f_n(v_0)+\sum_{s\leq {t\wedge \tau_{r}}} (f_n(V_{s-}+\Delta L_s)-f_n(V_{s-})). 
        \end{aligned}
        \]
        The term $\int_{t_0}^{t\wedge \tau_{r}} f_n'(V_s)F(V_s)s^{-\beta}\dd s$ is non-negative, since \eqref{hyp_sign_stable} holds. \\Hence, the previous equation can be written as
        \[f_n(V_{t\wedge \tau_{r}}) \leq f_n(v_0)+\sum_{s\leq {t\wedge \tau_{r}}} (f_n(V_{s})-f_n(V_{s-})).\]
    Since $\|f_n'\|_{\infty}\leq 1$, we deduce that $(f_n(V_s)-f_n(V_{s-}))\leq \abs{\Delta V_s}= \abs{\Delta L_s}$, hence,
    \[\abs{V_{t\wedge \tau_{r}}}\leq f_n(V_{t\wedge \tau_{r}})\leq f_n(v_0)+ \sum_{s\leq {t\wedge \tau_{r}}} \abs{\Delta L_s}. \]
    Furthermore, since $\kappa<\alpha<1$, we have
    \[\abs{V_{t\wedge \tau_{r}}}^{\kappa}\leq f_n(v_0)^{\kappa}+\left(\sum_{s\leq {t\wedge \tau_{r}}} \abs{\Delta L_s}\right)^{\kappa}. \]
    Taking the expectation, we get
    \[\mathbb{E}\left[\abs{V_{t\wedge \tau_{r}}}^{\kappa}\right]\leq \mathbb{E}\left[f_n(v_0)^{\kappa}\right]+ \mathbb{E}\left[\left(\sum_{s\leq t} \abs{\Delta L_s}\right)^{\kappa}\right]. \]
    Notice that the process $L_t^+:=\sum_{s\leq t} \abs{\Delta L_s}$ is an $\alpha$-stable process.
    Then, since $\kappa<\alpha$, letting $n\to +\infty$, we obtain
    \[ \mathbb{E}\left[\abs{V_{t\wedge \tau_{r}}}^{\kappa}  \right] \leq  \abs{v_0}^{\kappa}+ \mathbb{E}\left[ \abs{L_t^+}^{\kappa}\right] \leq C_{t_0,\kappa} t^{\frac{\kappa}{\alpha}}. \]
    Thanks to \cref{explosion_levy}, we can conclude that the explosion time of $V$ is a.s.\ infinite, and \eqref{eq:moment_alpha_1 } follows, letting $r\to \infty$.
\end{proof}

\begin{prop}\label{esperance2_levy} Pick $\alpha \in (1,2)$. Recall that $(V_t)_{t\geq t_0}$ is the solution of \eqref{equation time_levy}. For any $\gamma \in [0,1)$ and any $ \beta\in \mathbb{R}$, the explosion time $\tau_{\infty}$ is a.s.\ infinite and for all $\kappa \in [0,1]$, there exists $C_{\gamma,\kappa,\beta,t_0}$ such that we have
    \begin{equation}\label{eq: moment 1_alpha gamma_1}
        \forall t\geq  t_0, \ \mathbb{E}\left[ \abs{V_{t}}^{\kappa} \right]\leq C_{\gamma,\kappa,\beta,t_0}
        \begin{cases}t^{\frac{\kappa}{\alpha}} & \mbox{if }\frac{\gamma-1}{\alpha}+1\leq \beta, \\
            t^{\kappa\frac{1-\beta}{1-\gamma}} & \mbox{else.}
        \end{cases}
    \end{equation}
\end{prop}
\begin{remark}
    If $\frac{\gamma-1}{\alpha}+1>\beta$, then the moment estimate \eqref{eq: moment 1_alpha gamma_1} becomes \[\forall t \geq t_0, \ \mathbb{E}\left[ \abs{V_{t}}  \right]\leq C_{\gamma, \beta, t_0}t^{\frac{1-\beta}{1-\gamma}}.\]
\end{remark}
\begin{proof}[Proof of \cref{esperance2_levy}]
    Assume that $\gamma\in [0,1)$ and fix $\kappa \in [0,1]$. Then Jensen's inequality yields, for all $t\geq t_0$, $\mathbb{E}\left[\abs{V_{t}}^\kappa \right]\leq \mathbb{E}\left[\abs{V_{t}}\right]^\kappa$, hence it suffices to verify \eqref{eq: moment 1_alpha gamma_1} only for $\kappa=1$.\\
    Recall that under \eqref{hyp1_stable}, there exists a positive constant $K$, such that for all $v\in \mathbb{R}$, $\abs{F\left(v\right)}\leq K\abs{v}^{\gamma}$. Hence, we can write, for any $t\geq t_0$ and $r\geq0$,

    \[\begin{aligned}
            \abs{V_{(t \wedge \tau_r)-}} & \leq \abs{v_0-L_{t_0}}+\abs{L_{(t \wedge \tau_r)-}}+\int_{t_0}^{t \wedge \tau_r}s^{-\beta}\abs{F(V_{s \wedge \tau_r}) }\dd s \\ &\leq \abs{v_0-L_{t_0}}+\abs{L_{(t \wedge \tau_r)-}}+K\int_{t_0}^{t \wedge \tau_r}s^{-\beta}\abs{V_{s \wedge \tau_r} }^{\gamma}\dd s.
        \end{aligned}\]
    Since $L$ is an $\alpha$-stable process, it has a finite first moment, which can be computed. Taking the expectation in the above inequality, we get, by choosing $C_{t_0}$ big enough,
    \[\begin{aligned}
            \mathbb{E}\left[ \abs{V_{(t \wedge \tau_r)-}}\right] & \leq \mathbb{E}\left[\abs{v_0-L_{t_0}} \right] +\mathbb{E}\left[\abs{L_{(t \wedge \tau_r)-}} \right] +K\int_{t_0}^{t}s^{-\beta}\mathbb{E}\left[\abs{V_{s \wedge \tau_r} }^{\gamma}\right]\dd s \\ 
            & \leq C_{t_0}t^{\frac{1}{\alpha}}+K\int_{t_0}^{t}s^{-\beta}\mathbb{E}\left[\abs{V_{s \wedge \tau_r} }\right]^{\gamma}\dd s.
        \end{aligned} \]
    Recalling that $\tau_r$ is given by \eqref{stopping_time_levy}, the function $g_r:t\mapsto\mathbb{E}\left[ \abs{V_{(t \wedge \tau_r)-}} \right]$ is bounded by $r$.
    Applying a Grönwall-type lemma (see \cref{gronwall_levy}), we end up, for $\beta \neq 1$, with
    \[ \forall t \geq t_0, \ \mathbb{E}\left[ \abs{V_{(t\wedge \tau_r)-}}  \right]  \leq C_{\gamma} \left[C_{t_0}t^{\frac{1}{\alpha}}+\left(\dfrac{1-\gamma}{1-\beta}K(t^{1-\beta}-t_0^{1-\beta})\right)^{\frac{1}{1-\gamma}}  \right]. 
    \]
    The case $\beta=1$ can be treated similarly.
    Thanks to \cref{explosion_levy}, we conclude that the explosion time of $V$ is a.s.\ infinite, and \eqref{eq: moment 1_alpha gamma_1} follows from Fatou's lemma.
\end{proof}

\begin{prop}\label{esperance3_levy} Pick $\alpha\in [1,2)$. Assume here that for all $v\in \mathbb{R}$, $vF(v)\geq 0$.  Recall that $(V_t)_{t\geq t_0}$ is the solution of \eqref{equation time_levy}. For any $\gamma\in \mathbb{R}$ and any $ \beta\in \mathbb{R}$, the explosion time $\tau_{\infty}$ is a.s.\ infinite and there exists $C_{\kappa,t_0}$ such that
    \begin{equation}\label{eq: moment 1_alpha_levy}
        \text{for }\kappa\in(0,\alpha), \ \forall t\geq t_0,\ \mathbb{E}\left[\abs{V_{t}}^{\kappa}\right]\leq C_{\kappa,t_0}t^{\frac{\kappa}{\alpha}}.
    \end{equation}
\end{prop}

\begin{proof}
    The key idea is to slice the small and big jumps in a non-homogeneous way with respect to the function $\jsize \mapsto \jsize^{\frac{1}{\alpha}}$.  We write the proof in the general setting of $\alpha\in (1,2)$. When $\alpha=1$, the proof is similar since $\nu$ is symmetric.\\ Pick $\jsize\geq t_0$. As explained in \cite{ChaudruDeRaynalMultidimensionalstabledrivenStochastic2020} and references therein, by using this cutting threshold, the $\alpha$-stable Lévy driving process can be written as 
\[L_t-L_{t_0} = \int_{t_0}^t\int_{\abs{z}\leq \jsize^{\frac{1}{\alpha}}}z\widetilde{N}(\dd s, \dd z) +\int_{t_0}^t\int_{\abs{z}> \jsize^{\frac{1}{\alpha}}}z N(\dd s, \dd z) - \int_{t_0}^t\int_{\abs{z}> \jsize^{\frac{1}{\alpha}}} z\nu(\dd z)\dd s. \]
The two first integrals then satisfy the same scaling property as the $\alpha$-stable Lévy driving process.\\
We can compute
\[\int_{\abs{z}> \jsize^{\frac{1}{\alpha}}} z\nu(\dd z) =  \frac{a_+-a_-}{\alpha-1}\jsize^{\frac{1}{\alpha}-1}. \]
    \textsc{Step 1.} We first apply Itô's formula and estimate the expectation of each term for $\kappa\leq 1$, in order to get \eqref{eq: moment 1_alpha_levy}.\\
    Fix $\eta>0$ and define the $\mathcal{C}^2$-function $f:v\mapsto (\eta+v^2)^{\kappa/2}$.
    For all $t\geq t_0$, by Itô's formula, using that for all $v\in \mathbb{R}$, $vF(v)\geq 0$, we have
    \begin{equation}\label{eq:moment_Ito_levy}
        f(V_{t\wedge \tau_r})\leq f(V_0)-\frac{a_+-a_-}{\alpha-1}\jsize^{\frac{1}{\alpha}-1}\int_{t_0}^{t}\mathbb{1}_{\{s\leq \tau_r\} } f'(V_s)\dd s +M_{t}+R_{t}+S_{t},
    \end{equation}
    where
    \begin{equation}\label{eq: martingale_levy}
        M_t := \int_{t_0}^t \int_{0<\abs{z}<\jsize^{\frac{1}{\alpha}}}\mathbb{1}_{\{s\leq \tau_r\} } \left[f(V_{s-}+z)-f(V_{s-})\right] \widetilde{N}(\dd s, \dd z),
    \end{equation}
    \begin{equation}\label{eq: jumps_levy}
        R_t := \int_{t_0}^t \int_{\abs{z}\geq \jsize^{\frac{1}{\alpha}}}\mathbb{1}_{\{s\leq \tau_r\} } \left[f(V_{s-}+z)-f(V_{s-})\right]) N(\dd s, \dd z),
    \end{equation}
    \begin{equation}\label{eq: finite variation_levy}
        S_t := \int_{t_0}^t \int_{0<\abs{z}<\jsize^{\frac{1}{\alpha}}}\mathbb{1}_{\{s\leq \tau_r\} } \left[f(V_s+z)-f(V_s)-zf'(V_s)\right] \nu(\dd z) \dd s.
    \end{equation}
    Note that, since $\kappa<1$, for all $v\in \mathbb{R}$,
    \begin{equation}\label{eq: derivee_f_levy}
        \abs{f'(v)}\leq \kappa \eta^{\frac{\kappa-1}{2}}.
    \end{equation}
    Moreover, remark that for all $k>\alpha$,
    \begin{equation}\label{eq: levy_measure_petits_levy}
        \int_{0<\abs{z}<\jsize^{\frac{1}{\alpha}}} \abs{z}^{k} \nu(\dd z )  =  \frac{a_++a_-}{k-\alpha}\jsize^{\frac{k}{\alpha}-1},
    \end{equation}
    and for all $k<\alpha$,
    \begin{equation}\label{eq: levy_measure_ grands_levy}
        \int_{\abs{z}\geq\jsize^{\frac{1}{\alpha}}} \abs{z}^{k} \nu(\dd z )  =  \frac{a_++a_-}{\alpha-k}\jsize^{\frac{k}{\alpha}-1}.
    \end{equation}
    We estimate expectations of $M$, $R$ and $S$.\\
    To that end, we first show that the local martingale $(M_{t})_{t\geq t_0}$ is a martingale. 
    \\Fix $q\geq 2$ and $r\geq0$.
    Set 
    \[ I_t(q) := \int_{t_0}^{t} \int_{0<\abs{z}<\jsize^{\frac{1}{\alpha}}}\mathbb{1}_{\{s\leq \tau_r\} } \abs{f(V_{s-}+z)-f(V_{s-})}^q \nu(\dd z )\dd s.\]  Notice that, since for all $\abs{v}\leq r$ and $\abs{z}\leq \jsize^{\frac{1}{\alpha}}$, $\abs{f(v+z)-f(v)}\leq \|f'\mathbb{1}_{[-(r+\jsize^{\frac{1}{\alpha}}), r+\jsize^{\frac{1}{\alpha}}]}\|_{\infty}\abs{z}$, so we have 
    \[ I_t(q)\leq \|f'\mathbb{1}_{[-(r+\jsize^{\frac{1}{\alpha}}), r+\jsize^{\frac{1}{\alpha}}]}\|_{\infty}^q
            \int_{t_0}^{t} \int_{0<\abs{z}<\jsize^{\frac{1}{\alpha}}} \mathbb{1}_{\{s\leq \tau_r\} } \abs{z}^{q} \nu(\dd z )\dd s. \] 
            Hence, it is a finite quantity, since $q\geq 2$ and \eqref{eq: levy_measure_petits_levy} holds.
    Therefore, for $q\geq 2$, by Kunita's inequality (see \citer[Theorem 4.4.23 p. 265]{ApplebaumLevyprocessesstochastic2009}), there exists $D_q$ such that 
    \[
        \begin{aligned}
            \mathbb{E}\left[\sup_{t_0\leq s \leq t} \abs{M_s}^q\right] \leq D_q \left(\mathbb{E}\left[ I_t(2)^{q/2} \right] +\mathbb{E}\left[ I_t(q)\right]\right) <+\infty.
        \end{aligned}
    \]
    Hence, by \citer[Theorem 51 p. 38]{ProtterStochasticIntegrationDifferential2005}, $M$ is a martingale.
    \newline
   We estimate now the finite variation part $S$ defined in \eqref{eq: finite variation_levy}. We use a similar idea as in \citer[the proof of Theorem 3.1 p. 3863]{DengshiftHarnackinequalities2015}.
    Note that for all $v\in \mathbb{R}$,
    \[\begin{aligned}
            \abs{f''(v)} & = \kappa(2-\kappa)v^2(v^2+\eta)^{\frac{\kappa}{2}-2} +\kappa (v^2+\eta )^{\frac{\kappa}{2}-1}  \\& = \kappa(2-\kappa)
            v^2(v^2+\eta)^{-1}
            (v^2+\eta)^{\frac{\kappa}{2}-1} +\kappa (v^2+\eta )^{\frac{\kappa}{2}-1} \\& \leq \kappa(3-\kappa)(v^2+\eta )^{\frac{\kappa}{2}-1} \\& \leq \kappa(3-\kappa)\eta ^{\frac{\kappa}{2}-1}, \mbox{ since $\frac{\kappa}{2}-1<0$. }
        \end{aligned}\]
    Assume that $\abs{z}<\jsize^{\frac{1}{\alpha}}$.
    Using Taylor's formula, we get a.s.\,
    \[ \abs{f(V_s+z)-f(V_s)-zf'(V_s)} \leq \frac{1}{2} \kappa(3-\kappa)\eta ^{\frac{\kappa}{2}-1}z^2.\]
    Hence, we get the almost sure following bound
    \begin{equation*}
       \abs{ \int_{0<\abs{z}<\jsize^{\frac{1}{\alpha}}} \left(f(V_s+z)-f(V_s)-zf'(V_s)\right) \nu(\dd z)} \leq \frac{1}{2} \kappa(3-\kappa)\eta ^{\frac{\kappa}{2}-1} \int_{0<\abs{z}<\jsize^{\frac{1}{\alpha}}}z^2\nu(\dd z).
    \end{equation*}
    Injecting \eqref{eq: levy_measure_petits_levy}, we get
    \begin{equation}\label{deriveeseconde_levy}
        \abs{ \int_{0<\abs{z}<\jsize^{\frac{1}{\alpha}}} \left(f(V_s+z)-f(V_s)-zf'(V_s)\right) \nu(\dd z) } \leq \frac{1}{2} \kappa(3-\kappa)\eta ^{\frac{\kappa}{2}-1}\frac{a_++a_-}{2-\alpha}\jsize^{\frac{2}{\alpha}-1} .
     \end{equation}
    It remains to study the Poisson integral $R$ defined in \eqref{eq: jumps_levy}, using \citer[Theorem 2.3.7 p. 106]{ApplebaumLevyprocessesstochastic2009}. 
    Pick $\kappa\leq 1$, by Hölder property of power functions, we can write,
    \[\begin{aligned}
            f(v+z)-f(v)&=  \left(\eta +(v+z)^2\right)^{\frac{\kappa}{2}}-\left((v+z)^2\right)^{\frac{\kappa}{2}} + \left(v+z\right)^{\kappa} - v^{\kappa} +\left(v^2\right)^{\frac{\kappa}{2}} -\left(\eta+v^2\right)^{\frac{\kappa}{2}} \\& \leq 2\eta^{\frac{\kappa}{2}}+ \abs{z}^{\kappa}.
        \end{aligned}\]
    We deduce that
    \begin{equation*}
        \int_{\abs{z}\geq \jsize^{\frac{1}{\alpha}}} f(V_s+z)-f(V_s) \nu(\dd z) \leq \eta^{\frac{\kappa}{2}} \nu(\abs{z}\geq \jsize^{\frac{1}{\alpha}})+\int_{\abs{z}\geq\jsize^{\frac{1}{\alpha}}} \abs{z}^{\kappa} \nu(\dd z).
    \end{equation*}
    Injecting \eqref{eq: levy_measure_ grands_levy}, this leads to
    \begin{equation}\label{poissonintegral_levy}
        \int_{\abs{z}\geq \jsize^{\frac{1}{\alpha}}} f(V_s+z)-f(V_s) \nu(\dd z) \leq \eta^{\frac{\kappa}{2}} \frac{a_++a_-}{\alpha }\jsize^{-1}+\frac{a_++a_-}{\alpha-\kappa}\jsize^{\frac{\kappa}{\alpha}-1}.
    \end{equation}
      Gathering \eqref{eq: derivee_f_levy}, \eqref{poissonintegral_levy} and \eqref{deriveeseconde_levy}, we get
    \begin{multline}\label{eq:moment_step1a_levy}
        \mathbb{E}\left[\abs{V_{t\wedge \tau_r}}^{\kappa}\right]\leq \mathbb{E}\left[f(V_{t\wedge \tau_r})  \right] \leq \mathbb{E}\left[f(V_{t_0})\right]+t \jsize^{-1}\times \\  \left(\kappa \eta^{\frac{\kappa-1}{2}}\frac{a_+-a_-}{\alpha-1}\jsize^{\frac{1}{\alpha}}+\eta^{\kappa/2} \frac{a_++a_-}{\alpha }+\frac{a_++a_-}{\alpha-\kappa}\jsize^{\frac{\kappa}{\alpha}}+ \frac{1}{2} \kappa(3-\kappa)\eta^{\frac{\kappa}{2}-1} \frac{a_++a_-}{2-\alpha}\jsize^{\frac{2}{\alpha}}\right).
    \end{multline}
    Choosing $\eta=t^{\frac{2}{\alpha}}$ and $\jsize=t$, we get
    \[\begin{aligned}\label{eq:moment_step1_levy}
        \mathbb{E}\left[\abs{V_{t\wedge \tau_r}}^{\kappa}\right] &\leq \mathbb{E}\left[f(V_{t_0})\right]+t^{\frac{\kappa}{\alpha}} \times   \left(\kappa \frac{a_+-a_-}{\alpha-1}+ \frac{a_++a_-}{\alpha }+\frac{a_++a_-}{\alpha-\kappa}+ \frac{1}{2} \kappa(3-\kappa) \frac{a_++a_-}{2-\alpha}\right)\\&\leq C_{\kappa, t_0}t^{\frac{\kappa}{\alpha}}.
    \end{aligned}\]
    Thanks to \cref{explosion_levy}, we can conclude that the explosion time of $V$ is a.s.\ infinite and letting $r \to +\infty$, for all $\kappa\in [0,1],$
    \begin{equation}\label{eq: kappa_1_levy}
        \mathbb{E}\left[\abs{V_{t}}^{\kappa}\right] \leq C_{\kappa, t_0}t^{\frac{\kappa}{\alpha}}.
    \end{equation}
    \\
    \textsc{Step 2.} Pick $\kappa\in (1,\alpha)$. We estimate $R$ in another way, using again \citer[Theorem 2.3.7 p. 106]{ApplebaumLevyprocessesstochastic2009}.\\ 
    By Hölder property of power function and \eqref{eq: levy_measure_ grands_levy}, we get
    \begin{equation}\begin{aligned}\label{poissonintegralbis_stable}
        \int_{\abs{z}\geq \jsize^{\frac{1}{\alpha}}} \abs{f(V_s+z)-f(V_s)} \nu(\dd z) &\leq \int_{\abs{z}\geq \jsize^{\frac{1}{\alpha}}} \abs{2zV_s+z^2}^{\frac{\kappa}{2}}\nu(\dd z)\\&\leq C_{\kappa}\left(\frac{a_++a_-}{\alpha-\kappa}\jsize^{\frac{\kappa}{\alpha}-1}+\abs{V_s}^{\frac{\kappa}{2}}\frac{a_++a_-}{\alpha-\frac{\kappa}{2}}\jsize^{\frac{\kappa}{2\alpha}-1}\right).
    \end{aligned}
\end{equation}
    Gathering \eqref{deriveeseconde_levy}, \eqref{poissonintegralbis_stable} and then using that for all $v\in \mathbb{R}$, $\abs{f'(v)}\leq \kappa \abs{v}^{\kappa-1}$, 
   
    \begin{multline}\label{eq: moment intermediaire_levy}
        \mathbb{E}\left[\abs{V_{t\wedge \tau_r}}^{\kappa}\right] \leq \mathbb{E}\left[f(V_{t_0})\right]+t\left(C_{\kappa}\frac{a_++a_-}{\alpha-\kappa}\jsize^{\frac{\kappa}{\alpha}-1}+\frac{1}{2} \kappa(3-\kappa)\eta^{\frac{\kappa}{2}-1} \frac{a_++a_-}{2-\alpha}\jsize^{\frac{2}{\alpha}-1}\right)\\+ \kappa\frac{a_+-a_-}{\alpha-1}\jsize^{\frac{1}{\alpha}-1}\int_{t_0}^{t}\mathbb{E}\left[\abs{V_s}^{\kappa-1} \right]\dd s+ C_{\kappa}\frac{a_++a_-}{\alpha-\frac{\kappa}{2}}\jsize^{\frac{\kappa}{2\alpha}-1}\int_{t_0}^{t}\mathbb{E}\left[\abs{V_s}^{\frac{\kappa}{2}} \right]\dd s.
        \end{multline}
    Injecting \eqref{eq: kappa_1_levy}, and choosing $\eta=t^{\frac{2}{\alpha}}$ and $\jsize=t$, we get
    \begin{equation*}
        \mathbb{E}\left[\abs{V_{t\wedge \tau_r}}^{\kappa}\right] \leq C_{\kappa, t_0,\alpha}t^{\frac{\kappa}{\alpha}}.
    \end{equation*}
    Taking $r\to +\infty$, \eqref{eq: moment 1_alpha_levy} follows.
\end{proof}

\begin{ex}\label{linear_levy}
    Remark that the velocity process $V$ is more explicit in the linear case ($\gamma=1$).
    Choose $F(1)=\rho>0$, $F(-1)=-\rho$. Pick $\beta\neq 1$, so
    \[V_t=v_0+\exp\left(-\rho\frac{t^{1-\beta}}{1-\beta}\right)\int_{t_0}^t\exp\left(\rho\frac{s^{1-\beta}}{1-\beta}\right)\dd L_s
        \]
        is solution of \eqref{equation time_levy}.\\
    Hence, by an integration by parts,
    \[V_t=v_0+L_t-e^{\rho\frac{1}{1-\beta}(t_0^{1-\beta}-t^{1-\beta})}L_{t_0}-e^{-\rho\frac{t^{1-\beta}}{1-\beta}}\int_{t_0}^t\rho s^{-\beta}e^{\rho\frac{s^{1-\beta}}{1-\beta}}L_s \dd s.\]
    Thus,
    \[\mathbb{E}\left[\abs{V_t} \right]\leq C_{t_0}\left(t^{\frac{1}{\alpha}}+ t^{1-\beta+\frac{1}{\alpha}}\right)\leq C_{t_0}t^{\frac{1}{\alpha}}.\]
    The case $\beta=1$ can be treated similarly.
\end{ex}
\section{Proof of the asymptotic behavior of the solution} \label{s:proofs_levy}
This section is devoted to the proofs of our main results, Theorems \ref{main_thm_inh_levy} and \ref{main_thm_inh_levy_critique}.\\
Notice that it suffices to prove the convergence of the rescaled velocity process $(\eps^{\frac{1}{\alpha}}V_{t/\eps})_{t\geq \eps t_0}$ in the space ${\mathcal{D}}$ endowed with the Skorokhod topology. Assume for a moment that this convergence is proved.\\
For $\eps \in (0,1]$ and $t\geq \eps t_0$ 
we can write
\[\eps^{1+\frac{1}{\alpha}}X_{t/\eps} = \eps^{1+\frac{1}{\alpha}} x_0 +\int_{\eps t_0}^{t}V_s^{(\eps)} \dd s. \]
Let us introduce the mapping $g_{\eps}: V\mapsto \left(V_t, \int_{\eps t_0}^{t}V_s \dd s \right)_{t>0}$ defined and valued on ${\mathcal{D}}$.
Clearly, the theorem will be proved once we show that $g_{\eps}(V^{(\eps)}_{\bullet})$ converges weakly  in ${\mathcal{D}}$ endowed by the Skorokhod topology.
 This mapping is converging, as $\eps\to 0$, to the continuous mapping $g:V\mapsto \left(V_t, \int_{0}^{t}V_s \dd s \right)_{t>0}$.\\
In order to see $V^{(\eps)}$ as a process of $\mathcal{D}([0,+\infty))$, let us state for all $s\in [0,\eps t_0]$, $V_s^{(\eps)}:=V_{\eps t_0}^{(\eps)}=\eps^{\frac{1}{\alpha}}v_0 $. Call $P_{\eps}$, $P$ the distribution of $V^{(\eps)}$, $\mathcal{S}$, respectively. Invoking the Portmanteau theorem (see \citer[Theorem 2.1 p. 16]{BillingsleyConvergenceprobabilitymeasures1999}), it suffices to prove that for all bounded and uniformly continuous function $h:{\mathcal{D}}([0,+\infty))\times{\mathcal{D}}([0,+\infty)) \to \mathbb{R}$,
\[\int_{\mathcal{D}([0,+\infty))^2}h(g_{\eps}(\omega))\dd P_{\eps}(\dd \omega) \quad \underset{\eps \to 0}{\longrightarrow }\quad \int_{\mathcal{D}([0,+\infty))^2}h(g(\omega))\dd P(\dd \omega).  \]
Pick such a function $h$. By assumption, the convergence $P_{\eps}\underset{\eps \to 0}{\Longrightarrow} P$ holds, hence, using \cref{continous_mapping}, it suffices to prove that the uniformly bounded sequence $(h\circ g_{\eps})$ of continuous functions on $\mathcal{D}([0,+\infty))$ converges  to the continuous function $h\circ g$ uniformly on compact subsets of $\mathcal{D}([0,+\infty))$.
Let $K$ be a compact set of $\mathcal{D}([0,+\infty))$. Then, for all $\omega\in K$, $\max_{[0,\eps t_0]}\abs{\omega}$  is uniformly bounded by a constant, say $M$.
%
Fix $\eta>0$. By the uniform continuity of $h$, there exists $\delta>0$ such that for all $\omega \in K$,
\[\dd_u(g_{\eps}(\omega),g(\omega))\leq \delta \quad \Longrightarrow \quad \abs{h\circ g_{\eps}(\omega)-h\circ g (\omega)}\leq \eta.  \]
There exists $\eps_1>0$ small enough, such that for all $\eps\leq \eps_1$, for all $\omega\in K$,
\[\dd_u(g_{\eps}(\omega),g(\omega)) \leq C \abs{\int_0^{\eps t_0}\omega(s)\dd s}\leq C\eps t_0 M \leq \delta. \]
\\
Therefore, we proved that it suffices to prove the convergence of the rescaled velocity process $(\eps^{\frac{1}{\alpha}}V_{t/\eps})_{t\geq \eps t_0}$ in order to prove Theorems \ref{main_thm_inh_levy} and \ref{main_thm_inh_levy_critique}. \\
In Sections \ref{ss: proof_above_levy} and \ref{ss: proof_critical_levy}, the aim is to prove the convergence of the velocity process. 
\subsection{Asymptotic behavior above the critical line}\label{ss: proof_above_levy}

In the remainder of this section, we assume that $\gamma\geq0$ and $\beta >1+\frac{\gamma-1}{\alpha}$.

\begin{proof}[Proof of \cref{main_thm_inh_levy}]
    Thanks to a change of variables, we have, for all $\eps \in (0,1]$ and $t\geq \eps t_0$,
    \[\begin{aligned}
            \eps^{\frac{1}{\alpha}}V_{t/\eps}= & \eps^{\frac{1}{\alpha}}(v_0 - L_{t_0})+ \eps^{\frac{1}{\alpha}}L_{t/\eps} - \eps^{\frac{1}{\alpha}}\int_{t_0}^{t/\eps}F(V_{s})s^{-\beta}\dd s \\
            =                                         & \eps^{\frac{1}{\alpha}}(v_0 - L_{t_0})+ \eps^{\frac{1}{\alpha}}L_{t/\eps} - \eps^{\beta-1+\frac{1}{\alpha}}\int_{\eps t_0}^{t}F(V_{u/\eps})u^{-\beta}\dd u.
        \end{aligned}\]
By self-similarity, $L^{(\eps)}:=(\eps^{\frac{1}{\alpha}}L_{t/\eps})_{t\geq0}$ has the same distribution as an $\alpha$-stable process.
    As a consequence, thanks to \citer[Theorem 3.1 p. 27]{BillingsleyConvergenceprobabilitymeasures1999} and \cref{Skoro}, it suffices to prove
    \begin{equation}
        \label{proved_levy}
        \forall T\geq t_0 \, \sup_{\eps t_0\leq t \leq T} \abs{V_t^{(\eps)}-L_{t}^{(\eps)}}\stackrel{\mathbb{P}}{\to}0,\ \mbox{ as }\eps\to 0.
    \end{equation}
    Recall that under \eqref{hyp1_stable}, there exists a positive constant $K$, such that 
    \begin{equation}\label{major_levy}
        \eps^{\frac{\gamma}{\alpha}}\abs{F\left(\dfrac{V_{\bullet}^{(\eps)}}{\eps^{\frac{1}{\alpha}}}\right)}\leq K\abs{V_{\bullet}^{(\eps)}}^{\gamma}.
    \end{equation}
    Modifying the factor in front of the integral, we get
    \begin{equation}\label{eq:V and L_levy}
        V_t^{(\eps)} =\eps^{\frac{1}{\alpha}}(v_0 - L_{t_0})+ L_t^{(\eps)} - \eps^{\beta-1+\frac{1-\gamma}{\alpha}}\int_{\eps t_0}^{t}\eps^{\frac{\gamma}{\alpha}}F\left(\dfrac{V_{u}^{(\eps)}}{\eps^{\frac{1}{\alpha}}}\right)u^{-\beta}\dd u.
    \end{equation}
    Gathering \eqref{eq:V and L_levy} and \eqref{major_levy}, for all $T\geq \eps t_0$, we have,
    \[ \begin{aligned}
            \sup_{\eps t_0\leq t \leq T} \abs{V_t^{(\eps)}-L_t^{(\eps)}}\leq & \eps^{\frac{1}{\alpha}}(v_0 - L_{t_0})+\eps^{\beta-1+\frac{1-\gamma}{\alpha}}\sup_{\eps t_0\leq t \leq T}\abs{ \int_{\eps t_0}^{t}\eps^{\frac{\gamma}{\alpha}}F\left(\dfrac{V_{u}^{(\eps)}}{\eps^{\frac{1}{\alpha}}}\right)u^{-\beta}\dd u } \\ \leq & \eps^{\frac{1}{\alpha}}(v_0 - L_{t_0})+\eps^{\beta-1+\frac{1-\gamma}{\alpha}}\ \int_{\eps t_0}^{T}K\abs{V_{u}^{(\eps)}}^{\gamma}u^{-\beta}\dd u.
        \end{aligned}
    \]
    Taking the expectation and using the moment estimates on $V$ (see \cref{explosion_time_levy}), we obtain, when $\beta\neq \frac{\gamma}{\alpha}+ 1$,
    \[\begin{aligned}
        \eps^{\beta-1+\frac{(1-\gamma)}{\alpha}} \mathbb{E}\left[\int_{\eps t_0}^{T}K\abs{V_{u}^{(\eps)}}^{\gamma}u^{-\beta}\dd u\right] & = \eps^{\beta-1+\frac{1-\gamma}{\alpha}}\int_{\eps t_0}^{T}K\mathbb{E}\left[\abs{V_{u}^{(\eps)}}^{\gamma}\right]u^{-\beta}\dd u \\ & = \eps^{\beta-1+\frac{1}{\alpha}}\int_{\eps t_0}^{T}K\mathbb{E}\left[\abs{V_{u/\eps}}^{\gamma}\right]u^{-\beta}\dd u
            \\ &\leq  \eps^{\beta-1+\frac{1-\gamma}{\alpha}}\int_{\eps t_0}^{T}KC_{\alpha,\beta,t_0}u^{\frac{\gamma}{\alpha}-\beta}\dd u\\& =
            C\left( \eps^{\beta-1+\frac{1-\gamma}{\alpha}} T^{\frac{\gamma}{\alpha}-\beta+1}-t_0^{\frac{\gamma}{\alpha}-\beta+1}\eps^{\frac{1}{\alpha}}\right).
        \end{aligned} \]
    Hence, setting $q:= \min(\beta-1+\frac{1-\gamma}{\alpha},\frac{1}{\alpha} )$ which is positive, since $\beta >1+{\frac{\gamma-1}{\alpha}}$, we get \[\mathbb{E}\left[ \sup_{\eps t_0\leq t \leq T} \abs{V_t^{(\eps)}-L_t^{(\eps)}} \right] =  \underset{\eps \to 0}{O} (\eps^{q}).\]
    The case $\beta=1+\frac{\gamma}{\alpha}$ can be treated similarly to get
    \[\mathbb{E}\left[ \sup_{\eps t_0\leq t \leq T} \abs{V_t^{(\eps)}-L_t^{(\eps)}} \right]  = \underset{\eps \to 0}{O} (\eps^{\frac{1}{\alpha}}\ln(\eps)).\]
    This concludes the proof.
\end{proof}

\begin{remark}
    Observe that we did not use the condition ``$\gamma<1$ or for all $v\in \mathbb{R},\ vF(v)\geq 0$'' in this proof, except to get moment estimates.
\end{remark}
\subsection{Asymptotic behavior on the critical line}\label{ss: proof_critical_levy}
We adapt the Proposition 2.1, p. 187 of \cite{GradinaruExistenceasymptoticbehaviour2013} to the $\alpha$-stable Lévy case.\\
Pick a ${\mathcal{C}}^2$-diffeomorphism $\varphi: [0,t_1)\to [t_0,+\infty)$.
Let $V$ be the solution to the equation \eqref{equation time_levy}. Thanks to \citer[Proposition 3.4.1 p. 124]{SamorodnitskyStableNonGaussianRandom1994}, the following process is also an $\alpha$-stable process 
\begin{equation}\label{change_noise_levy}
    \left(R_t \right)_{t\geq0} := \left(\int_{0}^{t}\dfrac{\dd L_{\varphi(s)}}{\varphi'(s)^{\frac{1}{\alpha}}}\right)_{t\geq 0}.
\end{equation}
Then, by the change of variables $t=\varphi(s)$, we get
\[V_{\varphi(t)}-V_{\varphi(0)} = \int_{0}^{t}\varphi'(s)^{\frac{1}{\alpha}}\dd R_s-\int_{0}^{t}\dfrac{F(V_{\varphi(s)})}{\varphi(s)^{\beta}}\varphi'(s)\dd s.\]
Thanks to an integration by parts, we get
\[\dd \left(\dfrac{V_{\varphi(s)}}{\varphi'(s)^{\frac{1}{\alpha}}}\right) = \dd R_s-\dfrac{\varphi'(s)^{1-\frac{1}{\alpha}}}{\varphi(s)^{\beta}}F(V_{\varphi(s)})\dd s -\dfrac{\varphi''(s)}{\alpha\varphi'(s)}\dfrac{V_{\varphi(s)}}{\varphi'(s)^{\frac{1}{\alpha}}}\dd s.\]
Set $\Omega=\overline{\mathcal{D}}([t_0,\infty))$ the set of càdlàg functions, that equal $\infty$ after their (possibly infinite) explosion time. Introduce the scaling transformation $\Phi_{\varphi}$ defined, for $\omega \in \Omega $, by
\[\Phi_{\varphi}(\omega)(s) := \dfrac{\omega(\varphi(s))}{\varphi'(s)^{\frac{1}{\alpha}}}\text{, with }s\in [0,t_1).\] As a consequence, we obtain the following result.
\begin{prop}\label{changeoftime_levy}
    If $V$ is a solution to the equation \eqref{equation time_levy}, then $V^{(\varphi)}:=\Phi_{\varphi}(V)$ is a solution to
    \begin{equation}\label{change of time equation_levy}
        \dd V^{(\varphi)}_s=\dd R_s -\dfrac{\varphi'(s)^{1-\frac{1}{\alpha} }}{\varphi(s)^{\beta}}F(\varphi'(s)^{\frac{1}{\alpha}}V_s^{(\varphi)})\dd s -\dfrac{\varphi''(s)}{\varphi'(s)}\dfrac{V_s^{(\varphi)}}{\alpha}\dd s, \mbox{ with }  V_{0}^{(\varphi)}=\dfrac{V_{\varphi(0)}}{\varphi'(0)^{\frac{1}{\alpha}}},
    \end{equation} where $R$ is given by \eqref{change_noise_levy}.

    \noindent
    Conversely, if $V^{(\varphi)}$ is a solution to \eqref{change of time equation_levy}, then $\Phi_{\varphi}^{-1}(V^{(\varphi)})$ is a solution to the equation \eqref{equation time_levy}, where
    \[L_t-L_{t_0} := \int_{t_0}^{t}(\varphi'\circ \varphi^{-1})^{\frac{1}{\alpha}}(s)\dd R_{\varphi^{-1}(s)}.\]
    \noindent
    Furthermore, uniqueness in law, pathwise uniqueness, strong existence hold for the equation \eqref{equation time_levy} if and only if they hold for the equation \eqref{change of time equation_levy}.
\end{prop}
In the following, we will focus on an exponential change of time $\varphi_e: 0\leq t \mapsto t_0 e^t$. This scaling is convenient since it allows to produce a time-homogeneous term in \eqref{change of time equation_levy}.
Thanks to \cref{changeoftime_levy}, the process $V^{(e)}:=\Phi_e(V)$ satisfies the SDE driven by an $\alpha$-stable process $(R_t)_{t\geq0}$,

\begin{equation}\label{equation Ve_levy}
    \dd V_s^{(e)}=\dd R_s-\dfrac{V_s^{(e)}}{\alpha}\dd s- t_0^{1-\frac{1}{\alpha}-\beta} e^{(1-\frac{1}{\alpha}-\beta)s}F\big(t_0 ^{\frac{1}{\alpha}}e^{\frac{s}{\alpha}}V_s^{(e)}\big)\dd s.
\end{equation}
\begin{proof}[Proof of \cref{main_thm_inh_levy_critique}]
Assume in the sequel that $\beta=1+\frac{\gamma-1}{\alpha}$.\\
\textsc{Step 1.} Firstly we prove the finite-dimensional convergence. To that end, we reduce the problem to the convergence of a time-homogenous process. \\
Since \eqref{hyp1_stable} holds, \eqref{equation Ve_levy} becomes
\begin{equation}\label{equation_H_levy}
    \dd V_s^{(e)}=\dd R_s-\dfrac{V_s^{(e)}}{\alpha}\dd s-F\big(V_s^{(e)}\big)\dd s.
\end{equation}
Using the bijection $\Phi_e$ induced by the exponential change of time (see \cref{changeoftime_levy}), and the unique strong existence of the velocity process $V$ (see \cref{existence hyp1_stable} and \cref{explosion_time_levy}), there exists a pathwise unique strong solution $H$ to the time-homogeneous equation \eqref{equation_H_levy}. 
Hence, we have the equality \[ \left(\dfrac{V_{t_0e^t}}{(t_0e^{t})^{\nicefrac{1}{\alpha}}} \right)_{t\geq 0} =  (H_t)_{t\geq 0},\]
as two solutions of the same SDE, starting from the same point. 
We can write the above equality as \begin{equation*}\label{equal in law_levy}
    \left(\dfrac{V_{t}}{t^{\frac{1}{\alpha}}} \right)_{t\geq t_0} =  (H_{\log(t/t_0)})_{t\geq t_0}.
\end{equation*}
So, we have, for all $\eps>0$, $d\in \mathbb{N}^*$, and $(t_1, \cdots, t_d) \in [\eps t_0, +\infty)^d$,
\begin{equation}\label{eq: V and H_levy}
    \left(\dfrac{V_{\eps^{-1}t_1}}{(\eps^{-1}t_1)^{\nicefrac{1}{\alpha}}}, \cdots, \dfrac{V_{\eps^{-1}t_d}}{(\eps^{-1}t_d)^{\nicefrac{1}{\alpha}}} \right) =  \left( H_{\log(t_1)+\log((\eps t_0)^{-1})}, \cdots, H_{\log(t_d)+\log((\eps t_0)^{-1})} \right).
\end{equation}    
Since $\limsup_{\abs{x}\to +\infty} \frac{-F(x)-x/\alpha}{x}<0$, it follows from \citer[Proposition 0.1]{KulikExponentialergodicitysolutions2009} that the process $(H_t)_{t\geq 0}$ is exponentially ergodic. We denote its invariant measure by $\Lambda_F$.
Call $\widetilde{H}$ the solution of the time homogeneous equation \eqref{equation_H_levy} starting from $\Lambda_F$.
For $t_1, \cdots, t_d \in \mathbb{R}^d$, let $\Lambda_{F,t_1, \cdots, t_d}:= \mathcal{L}(\widetilde{H}_{t_1}, \cdots, \widetilde{H}_{t_d})$ 
be the distribution of $(\widetilde{H}_{t_1}, \cdots, \widetilde{H}_{t_d})$. Then, for all $s\geq 0$, $\Lambda_{F,t_1, \cdots, t_d}= \Lambda_{F,t_1+s, \cdots, t_d+s}$. Indeed, thanks to the invariance property of $\Lambda_F$, $(\widetilde{H}_{\bullet})$ and $(\widetilde{H}_{\bullet+s})$ satisfy the same SDE, starting from the same point. As a consequence, we get the stationary limit 
\begin{equation} \label{eq: stationnarity_levy}
    \lim\limits_{\eps\to 0}\mathcal{L}\left( \widetilde{H}_{\log(t_1)+\log((\eps t_0)^{-1})}, \cdots, \widetilde{H}_{\log(t_d)+\log((\eps t_0)^{-1})} \right) =   \Lambda_{F,\log(t_1), \cdots, \log(t_d)}. 
\end{equation} 
Moreover, by exponential ergodicity, for every $\psi: \mathbb{R}^d\to \mathbb{R}$ continuous and bounded function, 
\begin{equation} \label{eq: ergodicity_levy}
    \abs{\mathbb{E}\left[\psi\left(H_{\log(t_1/(t_0\eps))},\cdots, H_{\log(t_d/(t_0\eps))} \right)  \right]-\mathbb{E}\left[\psi\left(\widetilde{H}_{\log(t_1/(t_0\eps))},\cdots,  \widetilde{H}_{\log(t_d/(t_0\eps))} \right) \right] } \underset{\eps \to 0}{\longrightarrow}0.
\end{equation}
We postpone the proof of this limit in Step 2.
\\To conclude this step, gather \eqref{eq: V and H_levy}, \eqref{eq: stationnarity_levy} and \eqref{eq: ergodicity_levy} to get
\[\left(\dfrac{V_{\eps^{-1}t_1}}{(\eps^{-1}t_1)^{\nicefrac{1}{\alpha}}}, \cdots, \dfrac{V_{\eps^{-1}t_d}}{(\eps^{-1}t_d)^{\nicefrac{1}{\alpha}}} \right)\quad \underset{\eps\to 0}{\Longrightarrow} \quad \Lambda_{F,\log(t_1), \cdots, \log(t_d)}. \]
This can also be written as
\[\left(\eps^{\frac{1}{\alpha}}V_{t_1/\eps}, \cdots, \eps^{\frac{1}{\alpha}}V_{t_d/\eps} \right)\quad \underset{\eps\to 0}{\Longrightarrow}\quad  T*\Lambda_{F,\log(t_1), \cdots, \log(t_d)}, \]
where $T*\Lambda_{F,\log(t_1), \cdots, \log(t_d)}$ is the pushforward of the measure $\Lambda_{F,\log(t_1), \cdots, \log(t_d)}$ by the linear map $T(u_1, \cdots, u_d):=(t_1^{\nicefrac{1}{\alpha}}u_1, \cdots, t_d^{\nicefrac{1}{\alpha}}u_d)$.\\
\textsc{Step 2.} Let us now prove \eqref{eq: ergodicity_levy}. 
\\For sake of clarity, let us give a proof for $d=2$, the general case $d\geq2$ is similar.\\
Let $\psi: \mathbb{R}^2\to \mathbb{R}$ be a continuous and bounded function. 
Pick $\eps t_0 \leq s\leq t$. Set $h_0=v_0t_0^{-\frac{1}{\alpha}} $, \eqref{eq: ergodicity_levy} is now equivalent to 
\begin{equation*}
    \abs{\mathbb{E}\left[\psi\left(H_{\log(s/(t_0\eps))}, H_{\log(t/(t_0\eps))} \right) \Big | H_0=h_0 \right]-\mathbb{E}\left[\psi\left(H_{\log(s/(t_0\eps))}, H_{\log(t/(t_0\eps))} \right) \Big | H_0\sim \Lambda_F \right] } \underset{\eps \to 0}{\longrightarrow}0.
\end{equation*}
We set $\mu_{\eps}:= \mathcal{L}\left(H_{\log(s/(t_0\eps))} \Big| H_0=h_0 \right)$. 
We now use the generalized Markov property of solution to SDE driven by Lévy process. For the sake of completeness, we state and prove it in our context in Appendix (see \cref{Markov}).
This leads to  
\[
        \mathbb{E}\left[\psi\left(H_{\log(s/(t_0\eps))}, H_{\log(t/(t_0\eps))} \right) \Big | H_0=h_0 \right]
         =  \mathbb{E}\left[\psi\left(H_{0}, H_{\log(t/s)} \right)\Big | H_0\sim \mu_{\eps }\right]
  \]
  and, since $\Lambda_F$ is invariant,
  \[
    \mathbb{E}\left[\psi\left(H_{\log(s/(t_0\eps))}, H_{\log(t/(t_0\eps))} \right) \Big | H_0\sim \Lambda_F \right]
     =  \mathbb{E}\left[\psi\left(H_{0}, H_{\log(t/s)} \right)\Big | H_0\sim \Lambda_F\right].
\]
Then, we are reduced to prove 
\begin{equation*}
\abs{\mathbb{E}\left[\psi\left(H_{0}, H_{\log(t/s)} \right)\Big | H_0\sim \mu_{\eps }\right]-\mathbb{E}\left[\psi\left(H_{0}, H_{\log(t/s)} \right)\Big | H_0\sim \Lambda_F\right]}\underset{\eps \to 0}{\longrightarrow}0.
\end{equation*} 
The left-hand side can be written as,
\[
    \int_{\mathbb{R}}\mathbb{E}\left[\psi\left(H_{0}, H_{\log(t/s)} \right)\Big | H_0=y\right]\left(\mu_{\eps}(\dd y)- \Lambda_F(\dd y) \right). \]
    Hence, setting $p(t,x,\dd y):= \mathbb{P}_x(H_t\in \dd y)$ and $\|.\|_{TV}$ for the total variation norm, we get
    \begin{multline*}
        \abs{\mathbb{E}\left[\psi\left(H_{0}, H_{\log(t/s)} \right)\Big | H_0\sim \mu_{\eps }\right]-\mathbb{E}\left[\psi\left(H_{0}, H_{\log(t/s)} \right)\Big | H_0\sim \Lambda_F\right]}\\ \leq \left\|\psi \right\|_{\infty}\int_{\mathbb{R}}\abs{p\left(\log(s/(t_0\eps)),h_0, \dd y\right)-\Lambda_F(\dd y)} \\ \leq \left\|\psi \right\|_{\infty} \|p\left(\log(s/(t_0\eps)),h_0, \cdot\right)-\Lambda_F\|_{TV}.
    \end{multline*}
    This converges to 0, as $\eps \to 0$ by exponential ergodicity of $H$.
\\
\textsc{Step 3.} Let us prove now the tightness of the family of distributions of càdlàg process $\left(V^{(\eps)}\right)_{t\geq \eps t_0}=\left(\eps^{\frac{1}{\alpha}} V_{t/\eps} \right)_{t\geq \eps t_0}$ on every compact interval $[m,M]$, $0<m\leq M$. \\ We check the Aldous's tightness criterion stated at \citer[Theorem 16.10 p.178]{BillingsleyConvergenceprobabilitymeasures1999}.
Let $a,\ \eta,\ T$ be positive reals. Let $\tau$ be a discrete stopping time with finite range $\mathcal{T}$, bounded by $T$.
Choose $\delta>0$ and $\eps>0$ small enough.\\
We have, by Jensen's inequality, for $r=\frac{\alpha}{2}$,
\[\mathbb{E}\left[\abs{V_{\tau+\delta}^{(\eps)}-V_{\tau}^{(\eps)}}^r \right]\leq \mathbb{E}\left[\abs{L_{\tau+\delta}^{(\eps)}-L_{\tau}^{(\eps)}}^r \right] + \mathbb{E}\left[\int_{\tau}^{\tau+\delta}K\abs{V_u^{(\eps)}}^{\gamma}u^{-\beta}\dd u \right]^r.  \]
Since $L^{(\eps)}$ is an $\alpha$-stable process, by the strong Markov property,
\[\mathbb{E}\left[\abs{L_{\tau+\delta}^{(\eps)}-L_{\tau}^{(\eps)}}^r \right] = \mathbb{E}\left[ \mathbb{E}_{L_{\tau}}\left[\abs{L_{\delta}-L_{0}}^r\right] \right] \leq C \delta^{\frac{r}{\alpha}}.  \]
The stopping time has a finite range $\mathcal{T}$. Hence, we can write
\[ \begin{aligned}
        \mathbb{E}\left[\int_{\tau}^{\tau+\delta}K\abs{V_u^{(\eps)}}^{\gamma}u^{-\beta}\dd u \right] & = \mathbb{E}\left[\mathbb{E}\left[\int_{\tau}^{\tau+\delta}K\abs{V_u^{(\eps)}}^{\gamma}u^{-\beta}\dd u \Big|  \tau \right]\right] \\ & = \mathbb{E}\left[\sum_{\tau_i \in \mathcal{\tau}} \dfrac{1}{\mathbb{P}(\tau=\tau_i)}\mathbb{E}\left[ \mathbb{1}_{\{\tau=\tau_i\}}\int_{\tau_i}^{\tau_i+\delta}K\abs{V_u^{(\eps)}}^{\gamma}u^{-\beta}\dd u  \right]\mathbb{1}_{\{\tau=\tau_i\}}\right] \\ & \leq   \mathbb{E}\left[\sum_{\tau_i \in \mathcal{\tau}} \dfrac{1}{\mathbb{P}(\tau=\tau_i)}\mathbb{E}\left[\int_{\tau_i}^{\tau_i+\delta}K\abs{V_u^{(\eps)}}^{\gamma}u^{-\beta}\dd u  \right]\mathbb{1}_{\{\tau=\tau_i\}}\right].
    \end{aligned}
\]
For each $\tau_i\in \mathcal{T}$, using the relation $\beta=1+\frac{(\gamma-1)}{\alpha}$ and the moment estimates on $V$ (see \cref{explosion_time_levy}), we obtain 
\[\begin{aligned}
        \mathbb{E}\left[\int_{\tau_i}^{\tau_i+\delta}K\abs{V_u^{(\eps)}}^{\gamma}u^{-\beta}\dd u  \right] & = \int_{\tau_i}^{\tau_i+\delta} K \mathbb{E}\left[ \abs{V_u^{(\eps)}}^{\gamma} \right]u^{-\beta}\dd u \\ & \leq K  \int_{\tau_i}^{\tau_i+\delta} u^{\frac{\gamma}{\alpha}-\beta} \dd u  = K  \left[(\tau_i+\delta)^{\frac{1}{\alpha}} - \tau_i^{\frac{1}{\alpha}}\right] \\ & \leq K \delta ^{1,  \frac{1}{\alpha}}.
    \end{aligned}\]
The term $\delta^{1, \frac{1}{\alpha}}$ has to be read as $\delta$ or $\delta^{\frac{1}{\alpha}}$ depending on the fact that $x\mapsto x^{ \frac{1}{\alpha}}$ is a Lipschitz continuous function on $[0,T+\delta]$, if $\alpha<1$, or a $\frac{1}{\alpha}$-Hölder function, if $\alpha>1$.\\
By Markov's inequality, for $\delta$ small enough, we have
\[ \mathbb{P}\left(\abs{V_{\tau+\delta}^{(\eps)}- V_{\tau}^{(\eps)}}\geq a  \right) \leq \dfrac{K \delta ^{r, \frac{r}{\alpha}} }{a^r} \leq \eta.
\]
Furthermore, by moment estimates (see Propositions \ref{esperance_levy}, \ref{esperance2_levy} and \ref{esperance3_levy}), for all $t\geq \eps t_0$,
\[\sup_{\eps}\left[\abs{V_t^{(\eps)}}^r \right] \leq C t^{\frac{r}{\alpha}}. \]
Hence, using again Markov's inequality, by \citer[Corollary and Theorem 16.8 p. 175]{BillingsleyConvergenceprobabilitymeasures1999}, this concludes the proof of the tightness of the velocity process and therefore the proof of \cref{main_thm_inh_levy_critique}.

\end{proof}
\subsection{Asymptotic behavior below the critical line}
Assume in this section that $\beta < 1+\frac{\gamma-1}{\alpha}$ and $\alpha>1$. As a consequence, $\alpha q<1$. We take an interest into the power change of time $\varphi_q: t\mapsto \left(t_0^{1-\alpha q}+(1-\alpha q)t\right)^{\frac{1}{1-\alpha q}}$. Thanks to \cref{changeoftime_levy}, the process $V^{(q)}:= \Phi_q(V)$ satisfies the SDE driven by an $\alpha$-stable process $R$ having the same distribution as $L$,
\begin{equation}\label{equation Vq_levy}
    \dd V_s^{(q)}=\dd R_s- F\big(V_s^{(q)}\big)\dd s-q \varphi_q^{\alpha q -1 }V_s^{(q)}\dd s.
\end{equation}
\begin{proof}[Proof of \cref{sub-critical_levy}]
    \begin{steps}
        \item We first prove the finite dimensional convergence of the velocity process $(V_t^{(\eps)})_{t\geq \eps t_0}:= (\eps^q V_{t/\eps})_{t\geq \eps t_0}$. We give a proof for $d=2$, the general case $d\geq 2$ is similar.
        We call $H$ the ergodic process solution to 
        \begin{equation}\label{SDE_homogene_levy}
            \dd H_s =\dd R_s- F\Big(H_s\Big)\dd s,\mbox{ with } H_0=h_0:=v_0 t_0^{-q}.
        \end{equation}
        We denote by $\Pi_F$ its invariant measure.
        Using the bijection induced by the power change of time (\cref{changeoftime_levy}), as solutions of the same SDE starting at the same point, 
we have, for all $\eps>0$, and $(s,t)\in [\eps t_0,+\infty)^2$,
\begin{equation*}
\left(\eps^q \dfrac{V_{\eps^{-1}s}}{s^q}, \eps^q \dfrac{V_{\eps^{-1}t}}{t^q}   \right) =  \left(V^{(q)}_{\varphi^{-1}(\eps^{-1}s)},  V^{(q)}_{\varphi^{-1}(\eps^{-1}t)} \right)
\end{equation*}
Using \citer[Theorem 3.1 p. 27]{BillingsleyConvergenceprobabilitymeasures1999}, it suffices to prove that for all $(s,t)\in [\eps t_0,+\infty)^2$
\begin{itemize}
	\item $\absd{\left(\left(H_{\varphi^{-1}(\eps^{-1}s)},H_{\varphi^{-1}(\eps^{-1}t)} \right),\left(V^{(q)}_{\varphi^{-1}(\eps^{-1}s)},V^{(q)}_{\varphi^{-1}(\eps^{-1}t)} \right)  \right)}\underset{\eps \to 0}{\longrightarrow}0$, where $\absd{\cdot}$ is a metric on $\mathbb{R}^2$.
	\item $\left(H_{\varphi^{-1}(\eps^{-1}s)},H_{\varphi^{-1}(\eps^{-1}t)} \right)\quad \underset{\eps \to 0}{\Longrightarrow}\quad  \Pi_F \otimes \Pi_F $.
\end{itemize}
    \item We prove that for all $t\geq\eps t_0$, $\mathbb{E}\left[\left(H_{\varphi^{-1}(\eps^{-1}t)}- V^{(q)}_{\varphi^{-1}(\eps^{-1}t)} \right)^2 \right] \underset{\eps\to0}{\longrightarrow} 0$.
 \\
    Pick $t\geq \eps t_0$. For simplicity of notation, we write $H^{(\varphi,\eps)}_t:=H_{\varphi^{-1}(\eps^{-1}t)}$ and $V^{(\varphi,\eps)}_t:=V^{(q)}_{\varphi^{-1}(\eps^{-1}t)}$.
    We have
    \begin{equation*}
        \dd \left(H^{(\varphi,\eps)}_t- V^{(\varphi,\eps)}_t\right)= -\eps^{\alpha q-1}\left(F(H^{(\varphi,\eps)}_t)-F(V^{(\varphi,\eps)}_t) \right)t^{-\alpha q}\dd t+ qt^{-1}V^{(\varphi,\eps)}_t\dd t.
    \end{equation*}
    Pick $\delta>0$ and $\kappa\in (1,\alpha)$, by straightforward differentiation,
\begin{multline*}
	\dd \abs{H^{(\varphi,\eps)}_t- V^{(\varphi,\eps)}_t}^\kappa \leq  
	-\dfrac{\kappa\eps^{\alpha q-1}}{t^{\alpha q}}\abs{F(H^{(\varphi,\eps)}_t)-F(V^{(\varphi,\eps)}_t)}\abs{H^{(\varphi,\eps)}_t- V^{(\varphi,\eps)}_t}^{\kappa-1}\mathbb{1}_{\left\{\abs{H^{(\varphi,\eps)}_t- V^{(\varphi,\eps)}_t}>\delta\right\}}\dd t\\+\kappa t^{-1}qV^{(\varphi,\eps)}_t\sgn\left(H^{(\varphi,\eps)}_t- V^{(\varphi,\eps)}_t\right)\abs{H^{(\varphi,\eps)}_t- V^{(\varphi,\eps)}_t}^{\kappa-1}\dd t
\end{multline*}
Since $\gamma\geq1$, the function $F^{-1}$ is $\frac{1}{\gamma}$-Hölder, therefore there exists $C_{\gamma}>0$ such that, 
\begin{multline*}
	\dd \abs{H^{(\varphi,\eps)}_t- V^{(\varphi,\eps)}_t}^{\kappa} \leq  
	-\dfrac{\kappa \eps^{\alpha q-1}}{t^{\alpha q}}C_{\gamma}\delta^{\gamma-1}\abs{H^{(\varphi,\eps)}_t- V^{(\varphi,\eps)}_t}^{\kappa}\mathbb{1}_{\left\{\abs{H^{(\varphi,\eps)}_t- V^{(\varphi,\eps)}_t}>\delta\right\}}\dd t\\+\kappa t^{-1}\abs{qV^{(\varphi,\eps)}_t}\abs{H^{(\varphi,\eps)}_t- V^{(\varphi,\eps)}_t}^{\kappa-1}\dd t.
\end{multline*}
 We set $g_{\eps}(t)=\mathbb{E}\left[\abs{H^{(\varphi,\eps)}_t- V^{(\varphi,\eps)}_t
}^{\kappa} \right]$ and $\widetilde{g}_{\eps}(t)=\mathbb{E}\left[\abs{H^{(\varphi,\eps)}_t- V^{(\varphi,\eps)}_t
}^{\kappa}\mathbb{1}_{\left\{\abs{H^{(\varphi,\eps)}_t- V^{(\varphi,\eps)}_t}>\delta\right\}} \right]$. 
Taking expectation, we get,
\begin{equation}\label{ODE_levy}
    \widetilde{g}_{\eps}'(t)\leq -\dfrac{\kappa \eps^{\alpha q-1}}{t^{\alpha q}}C_{\gamma}\delta^{\gamma-1} \widetilde{g}_{\eps}(t)+b_{\eps}(t), \mbox{ with } \widetilde{g}_{\eps}(\eps t_0)=0
    \end{equation}
    where, using Hölder's inequality and moment estimates (\cref{esperance_levy}), we have 
    \[ \begin{aligned}
        b_{\eps}(t)&:=  \kappa t^{-1}q\mathbb{E}\left[V^{(\varphi,\eps)}_t\abs{H^{(\varphi,\eps)}_t- V^{(\varphi,\eps)}_t}^{\kappa-1}\right]\leq \kappa t^{-1}\abs{q}\mathbb{E}\left[\left(V^{(\varphi,\eps)}_t\right)^{\kappa}\right]^{\frac{1}{\kappa}}g_{\eps}(t)^{\frac{\kappa -1}{\kappa }}\\ & \leq C t^{-1}\left(\dfrac{t}{\eps}\right)^{\frac{1}{\alpha}-q}g_{\eps}(t)^{\frac{\kappa-1}{\kappa}}.
    \end{aligned}\]
    We use comparison theorem for ordinary differential equation on \eqref{ODE_levy} to get, setting $h(t):= \dfrac{1}{1-\alpha q}C_{\gamma}\delta^{\gamma-1}t^{1-\alpha q}$ 
\begin{equation*}
\widetilde{g}_{\eps}(t)\leq  \int_{\eps t_0}^t b_{\eps}(s)\exp(-\kappa \eps^{\alpha q-1}\left(h(t)-h(s)\right))\dd s.
\end{equation*}
As a consequence,
\[\begin{aligned}
	g_{\eps}(t)&\leq \delta^2+\widetilde{g}_{\eps}(t)
	\\ & \leq \delta^2+ \exp(-\kappa \eps^{\alpha q-1}h(t))C\int_{\eps t_0}^t s^{-1} \left(\dfrac{s}{\eps}\right)^{\frac{1}{\alpha}-q}\left[g_{\eps}(s)\exp(\kappa \eps^{\alpha q-1}h(s))\right]^{\frac{\kappa-1}{\kappa}} \exp(\eps^{\alpha q-1}h(s))\dd s.
\end{aligned}\]
Applying a Grönwall-type lemma (\cref{gronwall_levy}) to $g_{\eps}(s)\exp(\kappa \eps^{\alpha q-1}h(s))$, we have
\begin{equation*}
	\begin{aligned}
		g_{\eps}(t)&\leq C\delta^2+C\left(\int_{\eps t_0}^t s^{-1} \left(\dfrac{s}{\eps}\right)^{\frac{1}{\alpha}-q} \exp(-\eps^{\alpha q-1}(h(t)-h(s)))\dd s\right)^{\kappa}.
	\end{aligned}
\end{equation*}
We conclude, using the dominated convergence theorem, since $1-\alpha q>0$, that for all $\delta>0$
\begin{equation}\label{limsup_levy}
	0\leq \limsup_{\eps\to 0}g_{\eps}(t)\leq \delta^2.
\end{equation}
To prove the domination hypothesis, note that by optimization of the function $x\mapsto x^{\frac{1}{\alpha}}\exp(-Ax)$,
\[	\mathbb{1}_{\{\eps t_0\leq s \leq t\}}s^{-1} s^{-1} \left(\dfrac{s}{\eps}\right)^{\frac{1}{\alpha}-q} \exp(-\eps^{\alpha q-1}(h(t)-h(s)) \leq C \mathbb{1}_{\{0\leq s \leq t\}} s^{-1+\frac{1}{\alpha}-q} \dfrac{1}{\left(h(t)-h(s)\right)^{\frac{1}{\alpha}}}.
\]
This function in integrable, since $1-\alpha q>0$ and $\alpha>1$.
We let $\delta\to 0$ in \eqref{limsup_levy} to conclude that for all $t>0$, $\lim_{\eps\to 0}g_{\eps}(t)=0$.
\item Pick $(s,t)\in [\eps t_0,+\infty)^2$. Similarly, as in \cite{GradinaruAsymptoticbehaviortimeinhomogeneous2021a}, one can prove that the solution $H$ to \eqref{SDE_homogene_levy} satisfies
\begin{equation}\label{decouplage_H_levy}
	\left(H_{\varphi^{-1}(\eps^{-1}s)},H_{\varphi^{-1}(\eps^{-1}t)} \right)\quad \underset{\eps \to 0}{\Longrightarrow}\quad  \Pi_F \otimes \Pi_F. 
\end{equation}
\end{steps}
\end{proof}
\section{Extended results in the Lévy case}\label{s: extended}
In this section, the driving process of \eqref{equation time_levy} is supposed to be a general Lévy process $L$.
We denote by $(A,\nu, b)$ its generating triplet, with respect to the truncation function $h:z\mapsto -1\vee (z\wedge 1)$. Here  $\nu$ is the Lévy measure, $A>0$ and $b\in \mathbb{R}$.  This means that, in virtue of the Lévy-Khintchine formula (see \citer[Theorem 8.1 p. 37]{SatoLevyprocessesinfinitely1999}), the characteristic function of $L_t$ is given by
\[\mathbb{E}\left[e^{iuL_t}\right]= e^{t\psi(u)}, \mbox{ where } \psi(u):=-\dfrac{Au^2}{2}+ib u +\int_{\mathbb{R}}(e^{iux}-1-iuh(x))\nu(\dd x).  \]
Pick $\alpha\in (0,2)\setminus \{1\}$. The Lévy measure $\nu$ of the driving process is supposed to satisfy either 
\begin{equation}\tag{$H_1^{\nu,\alpha}$}\label{H1} 
    \begin{gathered}  
        \nu(z)=\frac{g(z)}{\abs{z}^{1+\alpha}}\mathbb{1}_{\{z\neq 0\}}, \mbox{ where $g$ is a non-negative measurable function such that}\\
        c^+:= \lim_{z\to +\infty}g(z)\geq 0, \quad c^-:=\lim_{z\to -\infty}g(z)\geq 0,
    \end{gathered}
\end{equation}
or 
\begin{equation}\tag{$H_2^{\nu,\alpha_{0}}$}\label{H2} \mbox{for some $\alpha_0>1$, }\int_{\abs{z}\geq 1}\abs{z}^{\alpha_0}\nu(\dd z) <+\infty .
\end{equation}
Remark that if $\nu$ satisfies \eqref{H1} with $\alpha>1$, then it satisfies \eqref{H2}.
Note that any tempered stable process satisfies \eqref{H1}, and any truncated $\alpha$-stable process satisfies \eqref{H2}.\\
For clarity, we recall the stochastic kinetic model:
for $t\geq t_0>0$,
\begin{equation}\tag{\text{SKE}}
    \dd V_t= \dd L_t-t^{-\beta}F(V_{t})\dd t, \mbox{ with } V_{t_0}=v_0>0,\
    \mbox{ and }\
    \dd X_t=V_t \dd t, \mbox{ with } X_{t_0}=x_0\in \mathbb{R}.
\end{equation}
We work under the assumption that there exists a unique solution to \eqref{equation time_levy}. 
We will show that \cref{main_thm_inh_levy} can be extended with this general Lévy driving process. We suppose first the Lévy process to be without a Brownian component. This case will be discussed in \cref{ss: withBM}. We obtain the two following theorems, depending on which hypothesis is satisfied by the Lévy measure $\nu$.
\begin{theo}\label{gene_thm_levy}
    Consider $\gamma\in[0,\alpha)$ and $\beta\geq 0$. Assume that \eqref{H1} and \eqref{hyp_sign_stable} are satisfied and define 
    \[ p_{\alpha}(\gamma) := \begin{cases}
        \frac{\gamma}{\alpha} & \mbox{if } g \mbox{ is bounded},\\
        \frac{\gamma}{\alpha} & \mbox{if } \alpha\in (0,1),\\
        \frac{\gamma}{\alpha} & \mbox{if } \alpha \in (1,2), \ \gamma \in [0,1) \mbox{ and } b=0, \\
        \gamma & \mbox{if } \alpha \in [1,2), \ \gamma \in [0,1) \mbox{ and } b\neq 0,\\
        \gamma & \mbox{if } \alpha \in (1,2), \ \gamma=1,\\
        \frac{\gamma}{\alpha}+\frac{\gamma}{2}& \mbox{if } \alpha \in (1,2), \ \gamma\in (1,\alpha).
    \end{cases}
        \]
Let $(V_t,X_t)_{t\geq t_0}$ be the solution to \eqref{equation time_levy}.
    \\Then there exist a rate of convergence $\eps^{\theta}$ and a Lévy process $\mathcal{L}$, given in \cref{limit_levy} \textit{(iii)-(v)}, such that, as $\eps\to 0$, if $\beta >1+p_{\alpha}(\gamma)-\theta$, in the space $\mathcal{D}$,
    \begin{equation*}
        (\eps^{\theta }V_{t/\eps}, \eps^{1+\theta}X_{t/\eps})_{t\geq \eps t_0}\quad \underset{\eps \to 0}{\Longrightarrow}\quad \left(\mathcal{L}_{t}, \int_{0}^t\mathcal{L}_s \dd s\right)_{t>0}.
    \end{equation*}
\end{theo}
\begin{theo}\label{gene_thm_levy2}
    Consider $\gamma \geq 0$ and $\beta \geq0$. Assume that \eqref{H2} is satisfied and define 
    \[ p_{\alpha_0}(\gamma):=\begin{cases}
        \gamma & \mbox{if } \gamma\in [0,1),\\
        \min(\gamma, \frac{\gamma}{\alpha_0}+\frac{\gamma}{2})\, & \mbox{if } \gamma \in [0,1) \mbox{ and } vF(v)\geq0, \\
        \frac{\gamma}{\alpha_0}+\frac{\gamma}{2} & \mbox{if } \gamma \in [1, \alpha_0] \mbox{ and } vF(v)\geq 0.
    \end{cases}
        \]
    Suppose that $\beta >1+p_{\alpha_0}(\gamma)-1$.

    Let $(V_t,X_t)_{t\geq t_0}$ be the solution to \eqref{equation time_levy}.
    \\Then there exists a Lévy process $\mathcal{L}$, given in \cref{limit_levy} \textit{(iii)}, such that, as $\eps\to 0$, in the space $\mathcal{D}$,
    \begin{equation*}
        (\eps V_{t/\eps}, \eps^{2}X_{t/\eps})_{t\geq \eps t_0}\quad \underset{\eps \to 0}{\Longrightarrow}\quad \left(\mathcal{L}_{t}, \int_{0}^t\mathcal{L}_s \dd s\right)_{t>0}.
    \end{equation*}
\end{theo}

\subsection{Large time behavior of the Lévy driving process}\label{ss: behav_levy}
Since the Lévy noise is no longer self-similar, we need to study its large-time behavior. We dedicate this subsection to the study of the convergence in distribution of the rescaled noise
$(L_t^{(\eps)})_{t\geq0}:= (v_{\eps}L_{t/\eps})_{t\geq0}$, for a suitable rate $v_{\eps}$, tending to 0.
\\
We introduce another assumption on $\nu$, which will sometimes be imposed in the sequel:
\begin{equation}\tag{$H_g$}\label{Hg}
    \begin{gathered}
        \nu( z)=\frac{g(z)}{\abs{z}^{1+\alpha}}\mathbb{1}_{\{z\neq 0\}}, \mbox{ where $g$ is a non negative measurable function with}\\
    \int_{1}^{+\infty}\dfrac{\abs{g(z)-g(-z)}}{z^{\alpha}}\dd x <+\infty.
    \end{gathered}
\end{equation}
Inspired from \cite{RosenbaumAsymptoticresultstimechanged2011}, we get the following result.
\begin{prop}\label{limit_levy} Let $(L_t)_{t\geq0}$ be a Lévy process having generating triplet $(A,\nu, b)$, with respect to the truncation function $h:z\mapsto -1\vee (z\wedge 1)$.
    \begin{enumerate}[label=(\roman*)]
        \item Assume that the Lévy measure satisfies the condition \eqref{H1} with $\alpha=1$, $c^+=c^-=c$ and \eqref{Hg}. Then the process $\left(\eps L_{t/\eps}\right)_{t\geq 0}$ converges in distribution to the strictly 1-stable Lévy process $\mathcal{L}$ generated by $(0,\nu^*,b^*)$, where $b^*:=b+\int_1^{+\infty} \frac{g(z)-g(-z)}{z^{1+\alpha}}\left(z-h(z)\right)\dd z$ and \[ \nu^*(\dd z) := \dfrac{c\mathbb{1}_{\{z\neq 0\}}}{\abs{z}^{2}}\dd z.\]
        \item Suppose that the Lévy measure satisfies the condition \eqref{H1} with $1<\alpha<2$ and \eqref{Hg}. If \[b^*:= b+\int_1^{+\infty} \frac{g(z)-g(-z)}{z^{1+\alpha}}\left(z-h(z)\right)\dd z \neq 0,\] then the process $\left(\eps L_{t/\eps}\right)_{t\geq 0}$  converges in distribution to the deterministic Lévy process $\mathcal{L}$ generated by $(0,0,b^*)$.
        \item If the Lévy measure satisfies \eqref{H2} and $b^* := b+\int_{\abs{z}\geq1} \left(z-h(z)\right)\nu(\dd z)\neq 0$, then the process $\left(\eps L_{t/\eps}\right)_{t\geq 0}$  converges in distribution to the deterministic Lévy process $L^*$ generated by $(0,0,b^*).$
        \item Assuming that $0<\alpha<1$, if the Lévy measure satisfies the condition \eqref{H1}, then the process $\left(\eps^{\frac{1}{\alpha}} L_{t/\eps}\right)_{t\geq 0}$ converges in distribution to the $\alpha$-stable Lévy process $\mathcal{L}$ with 
        \[\nu^*(\dd z) := \dfrac{c^+\mathbb{1}_{\{z>0\}}+c^-\mathbb{1}_{\{z<0\}}}{\abs{z}^{1+\alpha}}\dd z. \] 
        \item Assuming that $1<\alpha<2$ and $b=0$, if the Lévy measure satisfies the condition \eqref{H1}, then the process $\left(\eps^{\frac{1}{\alpha}} L_{t/\eps}\right)_{t\geq 0}$ converges in distribution to the Lévy process $\mathcal{L}$ with measure and center
        \[\nu^*(\dd z) := \dfrac{c^+\mathbb{1}_{\{z>0\}}+c^-\mathbb{1}_{\{z<0}\}}{\abs{z}^{1+\alpha}}\dd z, \quad  b^* := \int_{\mathbb{R}^*}z\nu^*(\dd z). \] 
    \end{enumerate}
\end{prop}
\begin{proof}
    From the Lévy-Khintchine's formula, the generating triplet of $(L^{(\eps)})_{t\geq 0} := (v_{\eps}L_{t/\eps})_{t\geq 0}$ is given by 
\begin{equation}\label{diff}
    A^{\eps}=\dfrac{v_{\eps}^2}{\eps}A,
\end{equation}
\begin{equation}\label{measure}
    \text{for all }B\in \mathcal{B}(\mathbb{R}), \ \nu^{\eps}(B)= \eps^{-1}\nu(\{z, zv_{\eps}\in B \}),
\end{equation}
\begin{equation}\label{drift}
    b^{\eps}=\dfrac{v_{\eps}}{\eps}\left[b+ \int_{\mathbb{R}^*}\left(v_{\eps}^{-1}h(v_{\eps}z)-h(z) \right) \nu(\dd z) \right].
\end{equation}
Call $(A^*, \nu^*, b^*)$ the generating triplet of the limiting process $\mathcal{L}$. By \citer[Corollary 3.6 p. 415]{JacodGeneralTheoryStochastic2003} and \citer[Theorem 14.7 p. 81]{SatoLevyprocessesinfinitely1999}, we have to check that
\begin{equation}\label{cond_drift}
    b^{\eps} \quad \underset{\eps \to 0}{\longrightarrow} \quad  b^*,
\end{equation}
\begin{equation}\label{cond_diff}
    A^{\eps}+\int_{\mathbb{R}^*}h^2(z)\nu^{\eps}(\dd z) \quad \longrightarrow \quad A^{*}+\int_{\mathbb{R}^*}h^2(z)\nu^{*}(\dd z),
\end{equation}
and that for any continuous and bounded function $f$ which is zero in a neighborhood of zero,
\begin{equation}\label{cond_measure}
    \int_{\mathbb{R}^*}f(z)\nu^{\eps}(\dd z) \quad \underset{\eps \to 0}{\longrightarrow}\quad \int_{\mathbb{R}^*}f(z)\nu^*(\dd z).
\end{equation}
\begin{enumerate}[label=(\roman*)]
    \item Recall that $\alpha=1$ and assume that the Lévy measure $\nu$ satisfies the conditions \eqref{H1} and \eqref{Hg}.\\
    To prove \eqref{cond_drift}, we write $b^{\eps}$ as

\[b^{\eps}=b+\int_0^{+\infty} \dfrac{g(z)-g(-z)}{z^{1+\alpha}}\left(\eps^{-1}h(\eps z)-h(z)\right)\dd z. \]
The dominated convergence theorem can be applied, and we show that it converges to
\[b^*=b+\int_1^{+\infty} \dfrac{g(x)-g(-x)}{x^{2}}\left(x-h(x)\right)\dd x.\]
Observe that the condition \eqref{Hg} was only required for this step.
\\Afterwards, note that, using a change of variables,

\[\begin{aligned}
    \int_{\mathbb{R}^*}h^2(z)\nu^{\eps}(\dd z) = v_{\eps}^{\alpha}\eps^{-1}\int_{\mathbb{R}^*}h^2(z)\dfrac{g(zv_{\eps}^{-1})}{\abs{z}^{1+\alpha}}\dd z,
\end{aligned} \]
and thus, we can apply the dominated convergence theorem to prove that the last integral converges to \[\int_{\mathbb{R}^*}h^2(z)\nu^{*}(\dd z).  \]
\\Let $f$ be a continuous and bounded function which is zero in a neighborhood of zero, then, using again a change of variables and applying the dominated convergence theorem,
\[\int_{\mathbb{R}^*}f(z)\nu^{\eps}(\dd z)\quad \underset{\eps \to 0}{\longrightarrow} \quad \int_{\mathbb{R}^*}f(z)\nu^{*}(\dd z). \]
%
    \item The proof is similar to the previous one and thus is left to the reader.
    \item In this point, we assume that \eqref{H2} holds.
    \\ The convergence \eqref{cond_drift} follows from the dominated convergence theorem. 
    \\ Using the explicit form of the truncation function, we get, for $v_{\eps}<1$,
    \begin{equation}\begin{aligned} \label{eq: hcarre}
        \int_{\mathbb{R}^*}h^2(z)\nu^{\eps}(\dd z)&= \int_{\mathbb{R}^*}\eps^{-1}h^2(v_{\eps}z)\nu(\dd z) \\
        &= \eps\int_{0<\abs{z}< 1}z^2\nu(\dd z)+\int_{1\leq \abs{z}\leq \eps^{-1}}\eps z^2\nu(\dd z) +\int_{\abs{x}> \eps^{-1}}\eps^{-1}\nu(\dd z).
    \end{aligned} \end{equation}
    Using the property of a Lévy measure, the first term in \eqref{eq: hcarre} converges to zero, as $\eps \to 0$, by assumption.
    Then, the last two terms in \eqref{eq: hcarre} are lower than
    \[\begin{aligned}
        \eps^{\alpha_0-1}\int_{1<\abs{z}< \eps^{-1}} \abs{z}^{\alpha_0}\nu(\dd z)+ \eps^{\alpha_0-1}\int_{\abs{z}\geq \eps^{-1}} \abs{z}^{\alpha_0}\nu(\dd z),
    \end{aligned}\] which converges to zero, when $\eps$ goes to zero, since $\alpha_0>1$.
    \\
    Let $f$ be a continuous and bounded function and assume that there exists $\delta>0$ such that $f(z)=0$ for all $\abs{z}\leq \delta$, thus,
\[ \int_{\mathbb{R}^*}f(z)\nu^{\eps}(\dd z)= \int_{\abs{z}>\frac{\eta}{\eps}}\eps^{-1}f(v_{\eps}z)\nu(\dd z)\leq 
    C\eps^{\alpha_0-1}\int_{\abs{z}>\frac{\eta}{\eps}}\abs{z}^{\alpha_0}\nu(\dd z).
\]
This vanishes as $\eps \to 0$.

    \item Take $0<\alpha<1$. 
    \\ Using the explicit form of $h$, giving in \cref{limit_levy}, we have,
    \[\begin{aligned}
        b^{\eps}&=\eps^{\frac{1}{\alpha}-1}b + \int_{-\infty}^{-\eps^{-\frac{1}{\alpha}}} [\eps^{\frac{1}{\alpha}-1}-\eps^{-1}]\nu(\dd z)+\int_{\eps^{-\frac{1}{\alpha}}}^{+\infty}[\eps^{-1}-\eps^{\frac{1}{\alpha}-1}]\nu(\dd z)\\& + \eps^{\frac{1}{\alpha}-1}\int_{-\eps^{-\frac{1}{\alpha}}}^{-1}(z+1)\nu(\dd z)+\eps^{\frac{1}{\alpha}-1}\int_{1}^{\eps^{-\frac{1}{\alpha}}}(z-1)\nu(\dd z).
    \end{aligned} \]
    Since $g$ has finite limits at infinity, for any $\delta>0$, we can choose $\eta>1$ big enough so that $\abs{g(z)-c^+}<\delta$ for $z\geq \eta$, and $\abs{g(z)-c^-}<\delta$ for $z\leq -\eta$. \\
    Hence,
    \[ \begin{aligned}
        \limsup_{\eps \to 0}\eps^{\frac{1}{\alpha}-1}\int_{1}^{\eps^{-\frac{1}{\alpha}}}(z-1)\nu(\dd z) &= \limsup_{\eps \to 0}\eps^{\frac{1}{\alpha}-1}\int_{\eta}^{\eps^{-\frac{1}{\alpha}}}(z-1)\nu(\dd z) \\ &\leq \limsup_{\eps \to 0} \eps^{\frac{1}{\alpha}-1}(c^++\delta)\int_{\eta}^{\eps^{-\frac{1}{\alpha}}}\dfrac{x-1}{x^{1+\alpha}}\dd x 
        =\dfrac{c^++\delta}{1-\alpha}.
    \end{aligned}\]
    Similarly,
    \[ \begin{aligned}
        \liminf_{\eps \to 0}\eps^{\frac{1}{\alpha}-1}\int_{1}^{\eps^{-\frac{1}{\alpha}}}(z-1)\nu(\dd z) &= \liminf_{\eps \to 0}\eps^{\frac{1}{\alpha}-1}\int_{\eta}^{\eps^{-\frac{1}{\alpha}}}(z-1)\nu(\dd z) \\ &\geq \liminf_{\eps \to 0} \eps^{\frac{1}{\alpha}-1}(c^+-\delta)\int_{\eta}^{\eps^{-\frac{1}{\alpha}}}\dfrac{z-1}{z^{1+\alpha}}\dd z 
        =\dfrac{c^+-\delta}{1-\alpha}.
    \end{aligned}\]
    The choice of $\delta$ being arbitrary, we get
    \[ \lim_{\eps \to 0}\eps^{\frac{1}{\alpha}-1}\int_{1}^{\eps^{-\frac{1}{\alpha}}}(z-1)\nu(\dd z) =
        \dfrac{c^+}{1-\alpha}.\]
        If $\eps$ is small enough, 
        then we can upper bound
        \[ \begin{aligned}
            \limsup_{\eps \to 0}\int_{\eps^{-\frac{1}{\alpha}}}^{+\infty}[\eps^{-1}-\eps^{\frac{1}{\alpha}-1}]\nu(\dd z) &\leq \limsup_{\eps \to 0}(\eps^{-1}-\eps^{\frac{1}{\alpha}-1})(c^++\delta)\nu([\eps^{-\frac{1}{\alpha}}, +\infty)) \\ &=
            \limsup_{\eps \to 0}(1-\eps^{\frac{1}{\alpha}})\dfrac{c^++\delta}{\alpha} =
            \dfrac{c^++\delta}{\alpha}.
        \end{aligned}\]
        Moreover,
        \[ \begin{aligned}
            \liminf_{\eps \to 0}\int_{\eps^{-\frac{1}{\alpha}}}^{+\infty}[\eps^{-1}-\eps^{\frac{1}{\alpha}-1}]\nu(\dd z) &\geq \liminf_{\eps \to 0}(\eps^{-1}-\eps^{\frac{1}{\alpha}-1})(c^+-\delta)\nu([\eps^{-\frac{1}{\alpha}}, +\infty]) \\ &=
            \liminf_{\eps \to 0}(1-\eps^{\frac{1}{\alpha}})\dfrac{c^+-\delta}{\alpha} =
            \dfrac{c^+-\delta}{\alpha}.
        \end{aligned}\]
        Similarly, this leads to
        \[\lim_{\eps \to 0}\eps^{\frac{1}{\alpha}-1}\int_{-\eps^{-\frac{1}{\alpha}}}^{-1}(z-1)\nu(\dd z) =-\dfrac{c^-}{1-\alpha}, \]
        and,
        \[ \lim_{\eps \to 0}\int_{-\infty}^{-\eps^{-\frac{1}{\alpha}}}[\eps^{\frac{1}{\alpha}-1}-\eps^{-1}]\nu(\dd z)=-\dfrac{c^-}{\alpha}. \]
        Hence, we obtain
        \[\lim_{\eps \to 0}b^{\eps} = \dfrac{c^+-c^-}{\alpha(1-\alpha)}. \]
        Since, $\int_{\mathbb{R}^*}h(z)\nu^{*}(\dd z)= \frac{c^+-c^-}{\alpha(1-\alpha)}$, the drift coefficient of the limiting process equals zero.
        \\  The proof of \eqref{cond_diff} and \eqref{cond_measure} are identical to the one done in \textit{(i)}.
    \item Take $1<\alpha<2$ and assume that $b=0$. 
    \\ After a change of variables, we can apply the dominated convergence theorem to
    \[\begin{aligned}
    b^{\eps}= \int_{\mathbb{R}^*}\left(h(y)-\eps^{\frac{1}{\alpha}}h(y\eps^{-\frac{1}{\alpha}})\right)\dfrac{g(y\eps^{-\frac{1}{\alpha}})}{\abs{y}^{1+\alpha}}\dd y.
\end{aligned}\]
The proof of \eqref{cond_diff} and \eqref{cond_measure} are identical to the one done in \textit{(i)}.
\end{enumerate}
\end{proof}

\subsection{Proofs of Theorems \ref{gene_thm_levy} and \ref{gene_thm_levy2}} \label{ss:proof extended}
In this section, we suppose that $\beta >1+p_{\alpha}(\gamma)-\theta$, where $\theta$ is the exponent of the rate of convergence given in \cref{limit_levy} and $p_{\alpha}(\gamma)$ is given in the statement of Theorems \ref{gene_thm_levy} and \ref{gene_thm_levy2}. Recall that $(V_t)_{t\geq t_0}$ is a solution of \eqref{equation time_levy}.
\subsubsection{Moment estimates of the velocity process}\label{subsect:moment}
As in \cref{s:moment_levy}, we will show that there exists a constant $C_{\gamma,\kappa,\beta,t_0}$ such that \begin{equation}
    \forall t\geq  t_0, \ \mathbb{E}\left[ \abs{V_{t}}^{\kappa} \right]\leq C_{\gamma,\kappa,\beta,t_0,b}t^{p_{\alpha}(\gamma,\kappa)},
\end{equation}
 where $p_{\alpha}(\gamma,\kappa)$ has to be detailed.
\begin{prop}\label{gene_esperance} Pick $\alpha\in (0,1)$. Assume \eqref{H1} and \eqref{hyp_sign_stable}. For any $\gamma$, $\beta$, the explosion time $\tau_{\infty}$ is a.s.\ infinite and for all $\kappa \in [0,\alpha)$, there exists $C_{\kappa,t_0,b}$ such that, we have
    \begin{equation*}
        \forall t\geq  t_0, \ \mathbb{E}\left[ \abs{V_{t}}^{\kappa} \right]\leq C_{\kappa,t_0,b}t^{\frac{\kappa}{\alpha}},  \text{ or equivalently } p_{\alpha}(\gamma,\kappa)=\frac{\kappa}{\alpha}.
    \end{equation*}
\end{prop}
\begin{proof}
    The proof is analogous to the proof of \cref{esperance_levy}. The term $\int_{t_0}^tbf'_n(V_s)\dd s$ is bounded by $\abs{b}t$ since $\|f_n'\|_{\infty}\leq 1$. Moreover, since $\nu$ satisfies \eqref{H1} with $\alpha<1$, the process $L_t^+:=\sum_{s\leq t} \abs{\Delta L_s}$ satisfies the conditions of Theorem 3.1 c) in \cite{DengshiftHarnackinequalities2015}. Thus, for all $\kappa\in [0,\alpha)$, for all $t\geq t_0$, 
    \[\mathbb{E}\left[\abs{L_t^+}^{\kappa} \right] \leq C_{t_0, \kappa}t^{\frac{\kappa}{\alpha}}.\]
    The estimates for $V$ follows.
\end{proof}

\begin{prop}\label{gene_esperance2} Assume either \eqref{H2} or \eqref{H1} with $\alpha\in (1,2)$. For any $\gamma \in [0,1)$ and any $ \beta\in \mathbb{R}$ the explosion time $\tau_{\infty}$ is a.s.\ infinite and for all $\kappa \in [0,1]$, there exists $C_{\gamma,\kappa,\beta,t_0}$ and $C_{\gamma, t_0}$ such that under \eqref{H1}, we have
    \begin{equation}\label{eq: moment 1_alpha gamma_1 bis}
        \forall t\geq  t_0, \ \mathbb{E}\left[ \abs{V_{t}}^{\kappa} \right]\leq C_{\gamma,\kappa,\beta,t_0}t^{\frac{\kappa}{\alpha}}+C_{\gamma, t_0}\abs{b}^{\kappa}t^{\kappa}.
    \end{equation}
    Or equivalently $p_{\alpha}(\gamma,\kappa)=\frac{\kappa}{\alpha}$ if $b=0$, and $p_{\alpha}(\gamma,\kappa)=\kappa$, else.\\
    And under \eqref{H2}, there exists $C_{\gamma,\kappa,\beta,t_0,b}$ such that we have
    \begin{equation*}
        \forall t\geq  t_0, \ \mathbb{E}\left[ \abs{V_{t}}^{\kappa} \right]\leq C_{\gamma,\kappa,\beta,t_0,b}t^{\kappa}.
    \end{equation*}
\end{prop}

\begin{proof}
    We explain the differences with respect to the proof of \cref{esperance2_levy}. Under each hypothesis on the Lévy measure $\nu$, the Lévy process has a first finite moment. 
    We write $L_t$ as the sum $bt+\widehat{L}_t$, where $\widehat{L}$ is the Lévy process without the drift part. \\We get,
   
    \[\mathbb{E}\left[ \abs{V_{(t \wedge \tau_r)-}}\right] \leq C_{t_0}\abs{b}t+\mathbb{E}\left[\abs{\widehat{L}_{(t \wedge \tau_r)-}} \right]+K\int_{t_0}^{t}s^{-\beta}\mathbb{E}\left[\abs{V_{s \wedge \tau_r} }\right]^{\gamma}\dd s.
     \]
    The proof of \citer[Theorem 3.1 (a) p. 3861]{DengshiftHarnackinequalities2015} can be adapted to estimate the moment of the Lévy process stopped at the stopping time $\tau_r$, given by \eqref{stopping_time_levy}. 
    \begin{enumerate}[label=(\roman*)]
        \item If the Lévy measure $\nu$ satisfies \eqref{H1} with $1<\alpha<2$, then it satisfies the conditions of Theorem 3.1 (a) and (c) in \cite{DengshiftHarnackinequalities2015}. Thus, for all $\kappa\in [0,1]$, there exists $C_{t_0, \kappa}$ such that, for all $t\geq t_0$,  
        \[\mathbb{E}\left[\abs{\widehat{L}_{(t\wedge \tau_r)-}}^{\kappa} \right] \leq C_{t_0, \kappa}t,\] 
        and, 
        \[\mathbb{E}\left[\abs{\widehat{L}_t}^{\kappa} \right] \leq C_{t_0, \kappa}t^{\frac{\kappa}{\alpha}}.\] 
        \item If the Lévy measure $\nu$ satisfies \eqref{H2}, then it satisfies the conditions of Theorem 3.1 (a) in \cite{DengshiftHarnackinequalities2015}. Thus, for all $\kappa\in [0,1]$, there exists $C_{t_0, \kappa}$ such that, for all $t\geq t_0$, 
        \[\mathbb{E}\left[\abs{\widehat{L}_{(t\wedge \tau_r)-}}^{\kappa} \right] \leq C_{t_0, \kappa}t.\]
    \end{enumerate}
    Applying the Grönwall-type lemma again (see \cref{gronwall_levy}) and Fatou's lemma, for $\beta\neq 1$, we end up with

    \[\forall t \geq t_0, \ \mathbb{E}\left[ \abs{V_{(t\wedge \tau_r)-}}  \right] \leq C_{\gamma} \left[C_{t_0,b}t+\left(\dfrac{1-\gamma}{1-\beta}K(t^{1-\beta}-t_0^{1-\beta})\right)^{\frac{1}{1-\gamma}}  \right].
\]
    The case $\beta=1$ can be done in a similar manner.\\
    We conclude that the explosion time of $V$ is a.s.\ infinite. To refine the estimates under the hypothesis \eqref{H1}, we apply again the Grönwall-type lemma to
    \[\mathbb{E}\left[ \abs{V_{t}}\right] \leq C_{t_0}\abs{b}t+\mathbb{E}\left[\abs{\widehat{L}_{t}} \right]+K\int_{t_0}^{t}s^{-\beta}\mathbb{E}\left[\abs{V_{s} }\right]^{\gamma}\dd s.
    \] 
 This proves \eqref{eq: moment 1_alpha gamma_1 bis} since $\frac{\gamma-1}{\alpha}+1-\beta\leq 0$.
\end{proof}

\begin{prop} Assume that for all $v\in \mathbb{R}$, $vF(v)\geq 0$. Pick $\gamma\in \mathbb{R}$ and $ \beta\in \mathbb{R}$.
    For each of the following cases, the explosion time $\tau_{\infty}$ is a.s.\ infinite.
    \begin{enumerate}[label=(\roman*)]
        \item Assume that \eqref{H1} holds with a bounded function $g$ and $\alpha\in [1,2)$, there exists $C_{\kappa,t_0}$ such that
        \begin{equation}\label{eq: moment 1_alphabis}
            \text{for }\kappa\in(0,\alpha), \ \forall t\geq t_0,\ \mathbb{E}\left[\abs{V_{t}}^{\kappa}\right]\leq C_{\kappa,t_0}t^{\frac{\kappa}{\alpha}}, \text{ or equivalently } p_{\alpha}(\gamma,\kappa)=\frac{\kappa}{\alpha}.
        \end{equation}
        \item Assume that \eqref{H1} holds with $\alpha\in [1,2)$. Then, \\
        for $\kappa\in[0,1]$ (resp. $\kappa\in[0,1)$, if $\alpha=1$), there exists $C_{\kappa,t_0}$ such that 
        \begin{equation}\label{eq: moment 1_alpha kappa_1bis}
            \forall t\geq t_0,\ \mathbb{E}\left[\abs{V_{t}}^{\kappa}\right]\leq C_{\kappa,t_0}t^{\kappa}, \text{ or equivalently } p_{\alpha}(\gamma,\kappa)=\kappa;
        \end{equation}
        for $\kappa\in(1,\alpha), \ \gamma \in[0,1) \text{ and }b=0$, there exists $C_{\kappa,t_0}$ such that
        \begin{equation}\label{eq: moment 1_alpha gamma_1 1_kappabis}
        \forall t\geq t_0,\ \mathbb{E}\left[\abs{V_{t}}^{\kappa}\right]\leq C_{\kappa,t_0}t^{\frac{3\kappa}{2\alpha}}, \text{ or equivalently } p_{\alpha}(\gamma,\kappa)=\frac{3\kappa}{2\alpha};
    \end{equation}
    for $\kappa\in(1,\alpha)$, there exists $C_{\kappa,t_0}$ such that
     \begin{equation}\label{eq: moment 1_alpha 1_kappabis}
        \forall t\geq t_0,\ \mathbb{E}\left[\abs{V_{t}}^{\kappa}\right]\leq C_{\kappa,t_0}t^{\frac{\kappa}{\alpha}+\frac{\kappa}{2}},\text{or equivalently } p_{\alpha}(\gamma,\kappa)=\frac{\kappa}{\alpha}+\frac{\kappa}{2}.
    \end{equation}
        \item Assume \eqref{H2}. Then,\\
        for $\kappa\in[0,1]$, there exists $C_{\kappa,t_0}$ such that,
         \begin{equation}\label{eq: moment 1_alpha kappa_1H2}
        \forall t\geq t_0,\ \mathbb{E}\left[\abs{V_{t}}^{\kappa}\right]\leq C_{\kappa,t_0}t^{\kappa}, \text{ or equivalently } p(\gamma,\kappa)=\kappa;
    \end{equation}
    for $\kappa\in[0,\alpha_0]$, there exists $C_{\kappa,t_0}$ such that 
    \begin{equation}\label{eq: moment 1_alpha 1_kappaH2}
       \forall t\geq t_0,\ \mathbb{E}\left[\abs{V_{t}}^{\kappa}\right]\leq C_{\kappa,t_0}t^{\frac{\kappa}{\alpha_0}+\frac{\kappa}{2}},\text{or equivalently } p(\gamma,\kappa)=\frac{\kappa}{\alpha_0}+\frac{\kappa}{2}.
    \end{equation}
    \end{enumerate}
\end{prop}
\begin{proof} We highlight only the differences with respect to the proof of \cref{esperance3_levy}. In the following, we assume that $\alpha \neq 1$, the proof is similar for $\alpha=1$. \\
    \textsc{Step A.} Assume that \eqref{H1} holds with a bounded function $g$.\\
    \textsc{Step A1.} Pick $\kappa\in [0,1]$. We adapt the estimates of the Itô's formula's terms. \\
    There is an additional term in \eqref{eq:moment_Ito_levy}, given, for $t\geq t_0$, by
    \[\int_{t_0}^{t}\mathbb{1}_{\{s\leq \tau_r\} } f'(V_s)b\dd s.\]
    If $g$ is a bounded function, the other terms of \eqref{eq:moment_Ito_levy} can be estimated in the same way, and \eqref{eq:moment_step1_levy} becomes
    \begin{multline}
        \mathbb{E}\left[\abs{V_{t\wedge \tau_r}}^{\kappa}\right] \leq \mathbb{E}\left[f(V_{t_0})\right]+t^{\frac{\kappa}{\alpha}+1-\frac{1}{\alpha}}\kappa \abs{b}+t^{\frac{\kappa}{\alpha}} \sup_{\mathbb{R}}\abs{g} \times \\  \left(\kappa \frac{a_+-a_-}{\alpha-1}+ \frac{a_++a_-}{\alpha }+\frac{a_++a_-}{\alpha-\kappa}+ \frac{1}{2} \kappa(3-\kappa) \frac{a_++a_-}{2-\alpha}\right)\leq C_{\kappa, t_0, b}t^{\frac{\kappa}{\alpha}}.
    \end{multline}
    This gives the proof of \eqref{eq: moment 1_alphabis} for $\kappa\in[0,1]$.\\
    \textsc{Step A2.} Pick $\kappa\in (1,\alpha)$. We estimate $R$, given by \eqref{eq: jumps_levy} in another way.\\ 
    The inequality \eqref{eq: moment intermediaire_levy} becomes 
    \begin{multline}\label{eq: moment intermediairebis}
        \mathbb{E}\left[\abs{V_{t\wedge \tau_r}}^{\kappa}\right]\leq \mathbb{E}\left[f(V_{t\wedge \tau_r})  \right] \leq \mathbb{E}\left[f(V_{t_0})\right]+t\left(C_{\kappa}\frac{a_++a_-}{\alpha-\kappa}\jsize^{\frac{\kappa}{\alpha}-1}+\frac{1}{2} \kappa(3-\kappa)\eta^{\frac{\kappa}{2}-1} \frac{a_++a_-}{2-\alpha}\jsize^{\frac{2}{\alpha}-1}\right)\\+ \kappa\left(\frac{a_+-a_-}{\alpha-1}\jsize^{\frac{1}{\alpha}-1}+\abs{b}\right)\int_{t_0}^{t}\mathbb{E}\left[\abs{V_s}^{\kappa-1} \right]\dd s+ C_{\kappa}\frac{a_++a_-}{\alpha-\frac{\kappa}{2}}\jsize^{\frac{\kappa}{2\alpha}-1}\int_{t_0}^{t}\mathbb{E}\left[\abs{V_s}^{\frac{\kappa}{2}} \right]\dd s.
        \end{multline}
    The inequality \eqref{eq: moment 1_alphabis} follows as in the proof of \cref{esperance3_levy}.\\
    \textsc{Step B.} When working under \eqref{H2}, we will pick $\kappa \leq \alpha_0$ during the proof.
   Assume that either \eqref{H1} holds with an unbounded function $g$ or \eqref{H2} holds, then we are not able to estimate the following terms
    \[\int_{\abs{z}> \jsize^{\frac{1}{\alpha}}} z\nu(\dd z) ,\quad \int_{0<\abs{z}<\jsize^{\frac{1}{\alpha}}} \abs{z}^{k} \nu(\dd z ), \quad \mbox{and}\quad \int_{\abs{z}\geq\jsize^{\frac{1}{\alpha}}} \abs{z}^{k} \nu(\dd z ).\]
    Hence, we apply the same proof scheme with small and big jumps sliced at 1. \\
    \textsc{Step B1.}
    This leads to a similar bound as in \eqref{eq:moment_step1_levy}:
    \begin{multline*}
        \mathbb{E}\left[\abs{V_{t\wedge \tau_r}}^{\kappa}\right]\leq \mathbb{E}\left[f(V_{t\wedge \tau_r})  \right] \leq \mathbb{E}\left[f(V_{t_0})\right]+\int_{t_0}^t \abs{b}\mathbb{E}\left[\abs{f'\left(V_{s\wedge \tau_r}\right)} \right] \dd s+ \mathbb{E}\left[t\wedge \tau_r\right] \times \\ \left(\eta^{\kappa/2} \nu(\abs{z}\geq 1)+\int_{\abs{z}\geq1} \abs{z}^{\kappa} \nu(\dd z)+ \frac{1}{2} \kappa(3-\kappa)\eta ^{\frac{\kappa}{2}-1} \int_{0<\abs{z}<1}z^2\nu(\dd z)\right) \leq C_{\kappa,t_0}t.
    \end{multline*}
    By Jensen's inequality, we can deduce \eqref{eq: moment 1_alpha kappa_1bis} and \eqref{eq: moment 1_alpha kappa_1H2}.\\
    \textsc{Step B2.} Pick $\kappa\in (1,\alpha)$. We estimate $R$, given by \eqref{eq: jumps_levy} in another way. \\ 
    By classical Hölder inequality,
    \begin{equation}\label{poissonintegralbis_levy}
        \int_{\abs{z}\geq 1} \abs{f(V_s+z)-f(V_s)} \nu(\dd z) \leq \int_{\abs{z}\geq 1} \abs{2zV_s+z^2}^{\frac{\kappa}{2}}\nu(\dd z)\leq C(1+\abs{V_s}^{\frac{\kappa}{2}})\int_{\abs{z}\geq 1}\abs{z}^{\kappa} \nu(\dd z)
    \end{equation}
    The last integral is finite. 
    Gathering \eqref{deriveeseconde_levy}, \eqref{poissonintegralbis_levy} and then using \eqref{eq: moment 1_alpha kappa_1bis}, \eqref{eq: moment 1_alpha kappa_1H2} or \cref{gene_esperance2},
       \[ \begin{aligned}
        \mathbb{E}\left[\abs{V_{t\wedge \tau_r}}^{\kappa}\right]\leq\mathbb{E}\left[f(V_{t\wedge \tau_r})  \right] \leq C_{\kappa, t_0}t + C\int_{t_0}^t
            \mathbb{E}\left[\abs{V_{s}}^{\frac{\kappa}{2}}\right]
            \dd s+ \int_{t_0}^t \abs{b}
            \mathbb{E}\left[\abs{V_{s}}^{\kappa-1} \right]
            \dd s.
    \end{aligned}\]
    Taking $r\to +\infty$, we can conclude that $p_{\alpha}(\gamma,\kappa)=1+p_{\alpha}(\gamma, \kappa/2)$.\\
    \textsc{Step B3.} We refine the estimates.
    Fix $\kappa\in[0,\alpha)$. There exists $\eps_1$ such that, $\kappa\leq \alpha-\eps$ and $\alpha-\eps > 1$.
    Hence, we can write
    \[\mathbb{E}\left[\abs{V_{t}}^{\alpha-\eps}\right]\leq C_{\kappa,t_0}t^{1+p_{\alpha}(\gamma, (\alpha-\eps)/2)}.\]
    Using Jensen's inequality, we get
    \[\mathbb{E}\left[\abs{V_{t}}^{\kappa}\right]\leq C_{\kappa,t_0}t^{\frac{\kappa}{\alpha-\eps}(1+p_{\alpha}(\gamma, (\alpha-\eps)/2))},\]
    and it suffices to let $\eps\to 0$ to conclude.
    This concludes the proof of \eqref{eq: moment 1_alpha gamma_1 1_kappabis} and \eqref{eq: moment 1_alpha 1_kappabis}
    The last step is identical under \eqref{H1} and can be done with $\alpha_0$ instead of $\alpha-\eps$, under \eqref{H2}. This concludes the proof of \eqref{eq: moment 1_alpha 1_kappaH2}.
\end{proof}

\subsubsection{Proofs of Theorems \ref{gene_thm_levy} and \ref{gene_thm_levy2}}
We are now in position to give the proofs of Theorems \ref{gene_thm_levy} and \ref{gene_thm_levy2}. Assume that either \eqref{H1} or \eqref{H2} is satisfied.\\
As in the $\alpha$-stable case (see \cref{s:proofs_levy}), it suffices to prove the convergence of the rescaled velocity process $(V_t^{(\eps)})_{t>0}:=(v_{\eps}V_{t/\eps})_{t> 0}$.\\
Thanks to \cref{limit_levy}, there exists $\theta\in \{ 1,\frac{1}{\alpha} \}$ such that $L^{(\eps)}:=(\eps^{\theta}L_{t/\eps})_{t\geq0}$ converges in distribution.\\
For sake of simplicity, we omit the dependencies of $p$ with respect to $\alpha$ and $\alpha_0$. Assume that $\gamma\geq0$ and $\beta >1+p(\gamma)-\theta$.\\
We can write, for $T\geq \eps t_0$, 
\begin{equation*}
    \sup_{\eps t_0\leq t \leq T} \abs{V_t^{(\eps)}-L_t^{(\eps)}}\leq  v_{\eps}(v_0 - L_{t_0})+v_{\eps}^{1-\gamma}\eps^{\beta-1}\ \int_{\eps t_0}^{T}K\abs{V_{u}^{(\eps)}}^{\gamma}u^{-\beta}\dd u.
\end{equation*}
Using the moment estimates on $V$ (see \cref{subsect:moment}), this leads, with $\beta \neq p(\gamma)+1$, to 
\begin{equation*}
    v_{\eps}^{1-\gamma}\eps^{\beta-1} \mathbb{E}\left[\int_{\eps t_0}^{T}K\abs{V_{u}^{(\eps)}}^{\gamma}u^{-\beta}\dd u\right] \leq  
            C\left( v_{\eps}\eps^{\beta-1-p(\gamma)} T^{p(\gamma)-\beta+1}-t_0^{p(\gamma)-\beta+1}v_{\eps}\right).
\end{equation*}
Hence, setting $q:= \min(\beta-1+\theta-p(\gamma), \theta)$, which is a positive number since $\beta >1+p(\gamma)-\theta$, we obtain \[\mathbb{E}\left[ \sup_{\eps t_0\leq t \leq T} \abs{V_t^{(\eps)}-L_t^{(\eps)}} \right] \underset{\eps \to 0}{=} O(\eps^{q}).\]
The case $\beta= 1 + p(\gamma)$ can be treated similarly. This concludes both of the proofs.

\subsection{Lévy driving process with a Brownian component} \label{ss: withBM}
We extend the results given in Theorems \ref{main_thm_inh_levy}, \ref{main_thm_inh_levy_critique}, \ref{gene_thm_levy} and \ref{gene_thm_levy2} to a driving process having a Brownian component. To this end, we decompose the Lévy noise as the sum of a Brownian part $(B_t)_{t\geq 0}$ and a Lévy part without a Brownian component $(L_t)_{t\geq 0}$. 
\begin{remark} The reasoning is the same with a Lévy process $\widehat{L} $ instead of the Brownian driving process. Indeed, let $\widehat{v}_{\eps}$ be the rate such that $(\widehat{v}_{\eps}\widehat{L}_{t/\eps})_{t\geq 0}$ converges. The conclusion follows, provided that $\widehat{v}_{\eps}^{-1}v_{\eps}$ tends to 0.
\end{remark}
Pick $\rho \geq0$. We consider the following one-dimensional stochastic kinetic model, for $t\geq t_0>0$,
\begin{equation}\label{SKE2}\tag{\text{SKE2}}
    \begin{gathered}
        \dd V_t^{(2)}= \dd L_t+\dd B_t-t^{-\beta}\rho \sgn(V_{t}^{(2)})\abs{V_{t}^{(2)}}^{\gamma}\dd t,\ V_{t_0}^{(2)}=v_0 \in \mathbb{R}, \\ 
    \mbox{and} \quad
    \dd X_t^{(2)}=V_t^{(2)} \dd t, \ X_{t_0}^{(2)}=x_0\in \mathbb{R}.
    \end{gathered}
\end{equation} 
We compare this solution with the solution of the equation driven by a Lévy process $L$ without a Brownian component. From  now on, for sake of convenience, the solutions of \eqref{equation time_levy} with $F(v)=\rho \sgn(v)\abs{v}^{\gamma}$, will be denoted by $(V_t^{(1)},X_t^{(1)})_{t\geq t_0}$ instead of $(V_t,X_t)_{t\geq t_0}$.\\
The asymptotic behavior of the latter SDE is given by Theorems \ref{main_thm_inh_levy}, \ref{main_thm_inh_levy_critique}, \ref{gene_thm_levy} and \ref{gene_thm_levy2}.
We will show that the Brownian part has no contribution.\\
Let us first point out some results about existence up to explosion of \eqref{SKE2}.
\begin{prop}~
    \begin{enumerate}[label=(\roman*)]
        \item When $\gamma \geq 1$, \eqref{SKE2} admits a unique solution up to explosion.
        \item Pick $\alpha\in(1,2)$. If \eqref{SKE2} is driven by the sum of a Brownian motion and an $\alpha$-stable process, then, if $0\leq \gamma<1$, there exists a solution of \eqref{SKE2}.
    \end{enumerate}
    
\end{prop}

\begin{proof}~
    \begin{enumerate}[label=(\roman*)]
        \item Since the drift coefficient is locally Lipschitz, by a standard localization argument, using  \citer[Theorem 3.1 and Remark 2 p. 338-339]{KunitaStochasticDifferentialEquations2004}, there is a unique solution defined up to explosion. 
        \item Since the drift coefficient is a continuous function, using a standard localization argument, we can apply \citer[Theorem 3.1 p. 866]{KurenokStochasticEquationsTimeDependent2007} to conclude.
    \end{enumerate}
\end{proof}
Assume in the following that \eqref{SKE2} admits a unique solution up to explosion.
\begin{theo} Consider $\gamma \in [0,\alpha)$ and $\beta\geq0$. Let $(V_t^{(2)},X_t^{(2)})_{t\geq t_0}$ be a solution to \eqref{SKE2}. The Lévy noise is supposed to be the sum of a Brownian noise $(B_t)_{t\geq0}$ and a Lévy process $(L_t)_{t\geq0}$ without Brownian component.
Suppose also that \eqref{equation time_levy} satisfies the conditions of either Theorems \ref{main_thm_inh_levy}, \ref{main_thm_inh_levy_critique}, \ref{gene_thm_levy}  or \ref{gene_thm_levy2}. 
Then $(V^{(2)},X^{(2)})$ has the same asymptotic behavior as $(V^{(1)},X^{(1)})$, which is given in Theorems \ref{gene_thm_levy} and \ref{gene_thm_levy2}. 
    
\end{theo}
\begin{proof}
    For $i\in \{1,2\}$, let us introduce the following rescaled processes:
    \[(V_t^{(i,\eps )})_{t\geq \eps t_0}:= (v_{\eps}V_{t/\eps}^{(i)})_{t\geq t_0}, \quad (X_t^{(i,\eps)})_{t\geq \eps t_0}:= (\eps v_{\eps}X_{t/\eps}^{(i)})_{t\geq t_0}, \quad \mbox{and}\quad (B_t^{(\eps)})_{t\geq0} := (\sqrt{\eps}B_{t/\eps})_{t\geq0}.\]
    We write, for all $t\geq \eps t_0$,
    \[V_t^{(1,\eps)} -V_t^{(2,\eps)}= v_{\eps}B_{t_0}-v_{\eps}\eps^{-\frac{1}{2}}B_t^{(\eps)} -v_{\eps}^{1-\gamma}\eps^{\beta-1}\int_{\eps t_0}^{t}v_{\eps}^{\gamma}\left(F\left(\dfrac{V_u^{(1,\eps)}}{v_{\eps}^{\gamma}} \right)-F\left(\dfrac{V_u^{(2,\eps)}}{v_{\eps}^{\gamma}} \right)  \right)u^{-\beta}\dd u.\]
    Hence, by Itô's formula, we get
    \begin{multline*}
        \left(V_t^{(1,\eps)} -V_t^{(2,\eps)}\right)^2= v_{\eps}^2B_{t_0}^2-2\int_{\eps t_0}^t \left( V_s^{(1,\eps)} -V_s^{(2,\eps)}\right)v_{\eps}\eps^{-\frac{1}{2}}\dd B_s^{(\eps)}+ v_{\eps}^2\eps^{-1}(t-\eps t_0)\\-2v_{\eps}^{1-\gamma}\eps^{\beta-1}\int_{\eps t_0}^t 
            \left( V_s^{(1,\eps)} -V_s^{(2,\eps)}\right)v_{\eps}^{\gamma}\left(F\left(\dfrac{V_u^{(1,\eps)}}{v_{\eps}^{\gamma}} \right)-F\left(\dfrac{V_u^{(2,\eps)}}{v_{\eps}^{\gamma}} \right)  \right)
            u^{-\beta}\dd u.
    \end{multline*}
    The last term on the right-hand side of the upper equality is positive since $F$ is an increasing function. \\
    Moreover, we can apply BDG inequalities (see \citer[Theorem 73 p. 222]{ProtterStochasticIntegrationDifferential2005}) to the local martingale, defined by
    \[M_t :=  -2\int_{\eps t_0}^t \left( V_s^{(1,\eps)} -V_s^{(2,\eps)}\right)v_{\eps}\eps^{-\frac{1}{2}}\dd B_s^{(\eps)}.\]
    This yields, for each $T>0$,
    \[\mathbb{E}\left[\sup_{\eps t_0\leq t \leq T}M_t^2 \right] \leq \mathbb{E}\left[ 4v_{\eps}^2\eps^{-1}\int_{\eps t_0}^T \left( V_s^{(1,\eps)} -V_s^{(2,\eps)}\right)^2\dd s \right].  \]
    For the sake of clarity, given $x$ and $y$, we introduce $x\lesssim y$ to mean that there exists a positive constant $C$ such that $x\leq Cy$.
    We get
    \[\begin{aligned}
        \mathbb{E}\left[\sup_{\eps t_0 \leq t \leq T}\left(V_t^{(1,\eps)} -V_t^{(2,\eps)}\right)^4 \right] & \lesssim  v_{\eps}^4\mathbb{E}\left[B_{t_0}^4\right] + v_{\eps}^4\eps^{-2}T^2+ \mathbb{E}\left[\sup_{\eps t_0\leq t \leq T}M_t^2  \right] \\ & \lesssim v_{\eps}^4\mathbb{E}\left[B_{t_0}^4\right] + v_{\eps}^4 \eps^{-2}T^2 + 4v_{\eps}^2\eps^{-1}\int_{\eps t_0}^T \mathbb{E}\left[\left( V_s^{(1,\eps)} -V_s^{(2,\eps)}\right)^2  \right]\dd s.
    \end{aligned}\]
    Hence, using that for all real $a$, $a^2\leq a^4+1$, we deduce that
    \begin{multline*}
        \mathbb{E}\left[\sup_{\eps t_0 \leq t \leq T}\left(V_t^{(1,\eps)} -V_t^{(2,\eps)}\right)^4 \right]  \\ \lesssim v_{\eps}^4\mathbb{E}\left[B_{t_0}^4\right] + v_{\eps}^4\eps^{-2}T^2 + 4v_{\eps}^2\eps^{-1}T +4v_{\eps}^2\eps^{-1}\int_{\eps t_0}^T \mathbb{E}\left[\sup_{\eps t_0 \leq t \leq s}\left( V_t^{(1,\eps)} -V_t^{(2,\eps)}\right)^4\right] \dd s.
    \end{multline*}
    Applying Grönwall's lemma, up to stopping times, we get
    \[ \mathbb{E}\left[\sup_{\eps t_0 \leq t \leq T}\left(V_t^{(1,\eps)} -V_t^{(2,\eps)}\right)^4 \right] \lesssim \left( v_{\eps}^4\mathbb{E}\left[B_{t_0}^4\right] + v_{\eps}^4\eps^{-2}T^2 + 4v_{\eps}^2\eps^{-1}T \right) \exp \left( 4v_{\eps}^2\eps^{-1} (T-\eps t_0) \right). \]
    We deduce also that
    \[ \begin{aligned}
        \mathbb{E}\left[\sup_{\eps t_0 \leq t \leq T}\left(X_t^{(1,\eps)} -X_t^{(2,\eps)}\right)^4 \right] &\leq T^4 \mathbb{E}\left[\sup_{\eps t_0 \leq t \leq T}\left(V_t^{(1,\eps)} -V_t^{(2,\eps)}\right)^4 \right]   \\ & \lesssim T^4\left( v_{\eps}^4\mathbb{E}\left[B_{t_0}^4\right] + v_{\eps}^4\eps^{-2}T^2 + 4v_{\eps}^2\eps^{-1}T \right) \exp \left( 4v_{\eps}^2\eps^{-1} (T-\eps t_0) \right). 
    \end{aligned}\]
    Hence, $\left(V_t^{(1,\eps)} -V_t^{(2,\eps)} , X_t^{(1,\eps)} -X_t^{(2,\eps)}\right)_{t\geq \eps t_0}$ converges uniformly towards 0 in probability on compacts, as $\eps$ goes to zero.\\
    The conclusion follows from Theorems \ref{gene_thm_levy}, \ref{gene_thm_levy2}, \citer[Theorem 3.1 p. 27]{BillingsleyConvergenceprobabilitymeasures1999} and \cref{Skoro}.
\end{proof}
\appendix
\section{Some technical results}
Let us state first a Grönwall-type lemma which has been used to get moment estimates. The proof can be found in \cite{GradinaruAsymptoticbehaviortimeinhomogeneous2021a}.
\begin{lem}[Grönwall-type lemma] \label{gronwall_levy}
    Fix $r\in [0,1)$ and $t_0\in \mathbb{R}$.
    Assume that $g$ is a non-negative real-valued function, $b$ is a positive function and $a$ is a differentiable real-valued  function. Moreover, suppose that the function $bg^{r}$ is a continuous function.\\ Assume that
    \begin{equation}\label{eq:debut_levy}
        \forall t \geq t_0, \ g(t)\leq a(t)+\int_{t_0}^{t}b(s)g(s)^{r}\dd s. 
    \end{equation}
    Then, setting $C_{r}:= 2^{\frac{1}{1-r}}$,
    \[\forall t \geq t_0, \ g(t)\leq C_{r} \left[a(t)+\left((1-r)\int_{t_0}^tb(s)\dd s\right)^{\frac{1}{1-r}}  \right].\] 
\end{lem}
\begin{remark}
    Call $H(t)$ the right-hand side of \eqref{eq:debut_levy}.
    If $g$ is not continuous, note that the function $H$
    is still continuous and satisfies \eqref{eq:debut_levy} (since $b$ is positive and $g\leq H$). So, one can apply the lemma to $H$ and thereafter use the inequality $g\leq H$. 
\end{remark}

We state now a technical lemma about the convergence in the spaces $\mathcal{C}$ and $\mathcal{D}$.
We recall that the space of continuous functions ${\mathcal{C}}$ is endowed with the uniform topology
\[\displaystyle \dd_u:f,g\in \mathcal{C}((0,+\infty))\mapsto \sum_{n=1}^{+\infty}\dfrac{1}{2^n}\min\Big(1,\sup_{[\frac{1}{n},n]}\abs{f-g}\Big).\]
Set $\Lambda:=\{\lambda:\mathbb{R}^+ \to \mathbb{R}^+, \text{continuous and increasing function such that } \lambda(0)=0, \ \lim\limits_{t\to +\infty}\lambda(t)=~+\infty  \}$ and  \[k_n(t)=\begin{cases}
    1     & \mbox{if } \frac{1}{n}\leq  t\leq n, \\
    n+1-t & \mbox{if } n<t<n+1,                  \\
    0     & \mbox{if } n+1\leq t.
\end{cases}  \]
The space of càdlàg functions ${\mathcal{D}}$ is endowed with the Skorokhod topology $\dd_s$ defined for $(f, g)\in \mathcal{D}((0,+\infty))^2$ by
\[
    \sum_{n= 1}^{+\infty}\dfrac{1}{2^n}\min\left(1, \inf\left\{ a \mbox{ s.t. } \exists \lambda\in \Lambda, \ \sup_{s\neq t}\abs{\log \dfrac{\lambda(t)-\lambda(s)}{t-s}}\leq a, \ \sup_{t\geq\frac1n}\abs{k_n(t)\left(f\circ \lambda(t)-g(t) \right)  } \leq a \right\}\right). \]
\begin{lem}\label{Skoro}~
    \begin{enumerate}[label=(\roman*)]
        \item The uniform distance is finer than the Skorokhod one i.e.\ $\dd_s\leq \dd_u$.
        \item Let $(f_{\eps})_{\eps\geq 0},(h_{\eps})_{\eps \geq 0}$ be two sequences of functions of $\mathcal{D}$. If for all $ n\geq 1$,
         \[\sup_{t\in[\frac{1}{n},n]}\abs{ f_{\eps}(t)-h_{\eps}(t) }\stackrel{\mathbb{P}}{\to}0\ \mbox{as }\eps\to 0,\] then $\dd (f_{\eps}, h_{\eps})\stackrel{\mathbb{P}}{\to}0$, where $\dd \in \{\dd_u, \dd_s \}$.
    \end{enumerate}
\end{lem}
\begin{proof}
    Let $f,g$ be two càdlàg functions. The first point is true using the definition of the metrics $\dd_s$ and $\dd_u$ and noting that
    \[\inf\left\{ a \mbox{ s.t. } \exists \lambda\in \Lambda, \ \sup_{s\neq t}\abs{\log \dfrac{\lambda(t)-\lambda(s)}{t-s}}\leq a, \ \sup_{t\geq\frac1n}\abs{k_n(t)\left(f(\lambda(t)-g(t) \right)  } \leq a \right\} \leq \sup_{[\frac{1}{n+1},n+1]}\abs{f-g}. \]
    Let us now prove the second part.
    Assume that for all $ n\geq 1$, $\sup_{[\frac{1}{n},n]}\abs{ f_{\eps}-h_{\eps} }\stackrel{\mathbb{P}}{\longrightarrow}0$, as $\eps \to 0$.
    Fix $\eta >0$ and choose $N>0$ such that $\sum\limits_{n=N+1}^{+\infty}\frac{1}{2^n}\leq \frac{\eta}{2}$. Then,
    \[
    \dd_s \left(f_{\eps},h_{\eps}\right) \leq \dd_u(f_{\eps},h_{\eps}) \leq \frac{\eta}{2}+ \sum_{n=1}^N\dfrac{1}{2^n} \sup_{[\frac{1}{n},n]}\abs{f_{\eps}-h_{\eps}}  .
    \]
    It follows, by setting $\eta':=\frac{\eta}{2}\left(\sum\limits_{n=1}^{+\infty}\frac{1}{2^n}\right)^{-1}$, that \[\mathbb{P}\left( \dd \left(f_{\eps},h_{\eps}\right) >\eta \right) \leq  \sum_{n=1}^{N}\mathbb{P}\left(\sup_{[\frac{1}{n},n]}\abs{ f_{\eps}-h_{\eps} }>\eta' \right)\underset{\eps \to 0}{\longrightarrow}0.\]
\end{proof}
For the sake of completeness, we state and improve the result of \citer[Problem 4.12 p. 64]{KaratzasBrownianMotionStochastic1998}, on a general metric space.
\begin{lem}\label{continous_mapping}
    Let $S$ be a Polish metric space endowed with a Borel $\sigma$-field $\mathcal{S}$. 
    Suppose that $(P_n)_{n\geq 1}$ is a sequence of probability measures on $(S,\mathcal{S})$ which converges weakly to a probability measure $P$. Suppose, in addition, that the sequence $(f_n)_{n\geq 1}$ of real-valued continuous functions on $S$ is uniformly bounded and converges to a continuous function $f$, the convergence being uniform on compact subsets of $S$. Then, we have
    \[{\lim_{n \to +\infty} \int_{S}f_n(\omega)\dd P_n(\omega) }=\int_{S}f(\omega)\dd P(\omega).   \]
\end{lem}
\begin{proof}
    Notice that, since $(P_n)_{n\geq 1}$ converges weakly thus, it is tight. 
    So, for each $\eps>0$, there exists a compact subset $K$ of $S$ such that for any $n\geq 1$, $P_n(K)\geq 1-\eps$.\\
    Let us decompose
    \[ \int_{S}f_n\dd P_n -\int_{S}f\dd P = A+B+C+D,  \]
    where
    \[            A  :=\int_{S\setminus K }f_n\dd P_n,       \quad 
            B :=\int_{ K }(f_n-f)\dd P_n,             \quad 
            C  :=\int_{S\setminus K }f\dd P_n,        \quad \mbox{and} \quad
            D := \int_{S }f\dd P_n -\int_{S }f \dd P.
    \]
    Let $M$ be a bound for the sequence $(f_n)$. Thus, by the choice of $K$,
    \[\abs{A}\leq  M P_n(S\setminus K)\leq M\eps.\] The third integral can be treated analogously.
    Besides, since the sequence $(f_n)$ converges uniformly on $K$ to $f$, there exists $n_{\eps}$ such that for all $n\geq n_{\eps}$, $\sup_K\abs{f_n-f}\leq \eps$. Thereby, we get
    \[\abs{B}\leq \eps P_n(K). \]
    The last integral is smaller than $\eps$ for $n$ large enough, since $P_n$ converges weakly to $P$, and this concludes the proof.
\end{proof}
\begin{remark} \cref{continous_mapping} could be applied with $S=\mathcal{C}([0,+\infty))$ or $\mathcal{D}([0,+\infty))$. However, the result for $S= \mathcal{C}([0,+\infty))$ is already contained in \citer[Problem 4.12 p. 64]{KaratzasBrownianMotionStochastic1998}.
\end{remark}
\begin{lem}\label{Markov}
    Let $(Y_t^{y})_{t\geq0}$ be the solution of a time-homogeneous SDE driven by an $\alpha$-stable process,
    \begin{equation}\label{eq: EDS_markov}
        \begin{cases}
            \dd Y_t= \dd L_t + b(Y_t)\dd t\\
            Y_{0}=y
        \end{cases}
    \end{equation}
    The measurable function $b$ is supposed to be such that \eqref{eq: EDS_markov} has a pathwise unique strong solution. Then $(Y_t)_{t\geq0}$ is a Markov process. \\Namely,
    for any $d\geq1$, $0\leq t_1\leq \cdots \leq t_d$, $u\geq0$ and any bounded measurable function $\phi: \mathbb{R}^d\to \mathbb{R}$,
    \begin{equation}\label{eq: markov}
        \mathbb{E}\left[\phi(Y_{t_1+u}^{y},\cdots,Y_{t_d+u}^{y} )\Big| \mathcal{F}_u \right]  =  \mathbb{E}\left[\phi(Y_{t_1}^{z},\cdots,Y_{t_d}^{z}) \right]_{z=Y_u^{y}}.
    \end{equation}
\end{lem}
\begin{proof} We give a proof for $d=2$, the general case being similar. Call $(Y_t^{s,y})$ the solution $ \dd Y_t= \dd L_t + b(Y_t)\dd t$, satisfying $Y_s=y$. 
    Let $\phi: \mathbb{R}^2\to \mathbb{R}$ be a bounded measurable function. Pick $u\geq 0$.
    Consider, for $y\in \mathbb{R}$ and $u\leq s \leq t$ the function
    \[G(y,s,t,u):= \left(Y_{s}^{u,y},Y_{t}^{u,y}\right)=\left(y+ L_s-L_u+\int_u^sb(Y_h)\dd h,\ y+ L_t-L_u+\int_u^tb(Y_h)\dd h \right). \]
    Pick $0\leq s\leq t$. Using pathwise uniqueness, $\left(Y_{s+u}^y,Y_{t+u}^y \right)= G\left(Y_u^y,s+u,t+u,u \right) $. Moreover, by time-homogeneity of the SDE, $(Y_{s+u}^{u,y})_{s\geq 0}$ and $(Y_s^{y})_{s\geq 0}$ have the same distribution. As a consequence, $G(y,s+u,t+u,u)=G(y,s,t,0)$. Besides, by Markov property of Lévy processes, the function $G$ is independent of $\mathcal{F}_u$. Hence, 
    \[\begin{aligned}
        \mathbb{E}\left[\phi(Y_{s+u}^{y},Y_{t+u}^{y} )\Big| \mathcal{F}_u \right]&= \mathbb{E}\left[\phi\circ G\left(Y_u^y,s+u,t+u,u \right)\Big| \mathcal{F}_u \right] = \mathbb{E}\left[\phi\circ G\left(z,s,t,0 \right)\Big| \mathcal{F}_u \right]_{z=Y_u^y} \\ &= \mathbb{E}\left[\phi\circ G\left(z,s,t,0 \right)\right]_{z=Y_u^y}\\&= \mathbb{E}\left[\phi(Y_{s}^{z},Y_{t}^{z}) \right]_{z=Y_u^{y}}
    \end{aligned}\]
    This concludes the proof. 
\end{proof}
{\bf Acknowledgements}
The authors would like to thank Nicolas Fournier and Thomas Cavallazzi for helpful discussions about moment estimates. We would like to thank Vlad Bally for his suggestion which gives rise to the \cref{s: extended}.
We are grateful to Thomas Cavallazzi for his careful reading of the manuscript.
\bibliographystyle{alpha}
\bibliography{kinetic_jumps}
\end{document}